\documentclass[12pt]{amsart}
\usepackage[margin=1in,letterpaper,portrait]{geometry}
\usepackage{ytableau}
\ytableausetup{notabloids}
\ytableausetup{centertableaux}
\usepackage{latexsym,amsmath,amssymb,verbatim, tikz}
\usepackage{graphicx}
\usepackage[colorlinks]{hyperref}
\usepackage{mathrsfs}
\usepackage{amsfonts}
\usepackage{amsthm,youngtab}
\usepackage{graphicx}
\usepackage{caption}
\usepackage{enumerate,enumitem}
\usepackage{tikz-cd}
\usepackage{rotating}
\usepackage[all]{xy}
\usepackage[titletoc, title]{appendix}
\usepackage{amscd}
\usepackage{float}
\usepackage{mathtools}
\hypersetup{
bookmarksnumbered,
pdfstartview={FitH},
breaklinks=true,
linkcolor=blue,
urlcolor=blue,
citecolor=blue,
bookmarksdepth=2
}	
\usepackage{caption}
\usepackage[utf8x]{inputenc}
\usepackage{adjustbox}
\usepackage[numbers]{natbib} 
\usepackage{dynkin-diagrams}

\theoremstyle{plain}
   
   \newtheorem{theorem}{Theorem}[section]
   \newtheorem{proposition}[theorem]{Proposition}
   
   \newtheorem{lemma}[theorem]{Lemma}
   \newtheorem{corollary}[theorem]{Corollary}

\theoremstyle{definition}
   \newtheorem{definition}[theorem]{Definition}
   
   \newtheorem{example}[theorem]{Example}

   \newtheorem{remark}[theorem]{Remark}

\numberwithin{equation}{section}

\newcommand{\CC}{{\mathbb {C}}}
\newcommand{\QQ}{{\mathbb {Q}}}

\newcommand{\ZZ}{{\mathbb {Z}}}
\newcommand{\NN}{{\mathbb {N}}}

\newlength{\mysizetiny}
\newlength{\mysizesmall}
\newlength{\mysize}
\newlength{\mysizelarge}
\newcommand{\hg}{\hat{\mathfrak{g}}}
\newcommand{\hgs}{\hat{\mathfrak{g}}^{\sigma}}

\newcommand{\lgs}{\mathcal{L} \mathfrak{g} ^{\sigma}}
\newcommand{\bs}{\mathfrak{b}^{\sigma}}
\newcommand{\gs}{\mathfrak{g}^\sigma}
\newcommand{\g}{{\mathfrak{g}}}
\newcommand{\U}{\mathcal{U}}

\newcommand{\mo}{\mathcal{O}}

\begin{document}
	
\title{Representations of shifted twisted  quantum affine algebras}
\author{Fei-Fei Li, Jian-Rong Li, and Yan-Feng Luo$^\dag$}\thanks{$\dag$ Corresponding author}

\address{Fei-Fei Li: School of Mathematics and Statistics, Lanzhou University, Lanzhou 730000, P. R. China.}
\email{liff2025@lzu.edu.cn}

\address{Jian-Rong Li: Faculty of Mathematics, University of Vienna, Oskar-Morgenstern Platz 1, 1090 Vienna, Austria.}
\email{lijr07@gmail.com}

\address{Yan-Feng Luo: School of Mathematics and Statistics, Lanzhou University, Lanzhou 730000, P. R. China.}
\email{luoyf@lzu.edu.cn}
\date{}
	
\begin{abstract}
In this paper, we introduce and study shifted twisted quantum affine algebras which provide a twisted counterpart of the theory of shifted quantum affine
algebras. The shifted
twisted quantum affine algebra $\U_q^{\mu_+,\mu_-}(\hgs)$ is obtained from the Drinfeld current presentation of twisted quantum loop algebras by shifting the Cartan--Drinfeld currents $\phi_i^\pm(z)$ according to a coweight pair $(\mu_+,\mu_-)$. 
We prove that it admits a triangular decomposition and that,
up to isomorphism, they depend only on the total shift $\mu=\mu_++\mu_-$. For each shift $\mu$,
we define a category $\mathcal O_\mu$ of representations of  $\U_q^\mu(\hgs) = \U_q^{0,\mu}(\hgs)$ and prove a rationality theorem for the
Cartan currents: on every weight space, the two currents $\phi_i^+(z)$ and
$\phi_i^-(z)$ are expansions of the same rational operator-valued function,
whose degree is prescribed by $\alpha_i(\mu)$. As a consequence, we classify
the simple objects of $\mathcal O_\mu$ by rational $\ell$-weights of the
corresponding degrees. We then construct a deformed Drinfeld coproduct and use it to define a fusion
product on the direct sum $\mathcal{O}^{sh}$ of the categories $\mathcal O_\mu$. This fusion
product is compatible with $q$-characters. We also classify finite-dimensional simple modules in $\mathcal{O}^{sh}$ in terms of dominant rational $\ell$-weights, with a separate treatment
of type $A_{2n}^{(2)}$. Finally, we construct 
restriction representations relating representations of twisted quantum affine Borel algebras to representations of shifted twisted quantum affine algebras, and establish a $q$-characters formula for simple finite-dimensional representations of  shifted twisted quantum affine algebras in terms of the $q$-characters of the corresponding simple representations of the twisted quantum affine Borel algebra $\U_q(\bs)$.

\hspace{0.15cm}

\noindent
{\bf Key words}: Shifted twisted  quantum affine algebras; finite-dimensional modules; $q$-characters; fusion product

\hspace{0.15cm}

\noindent
{\bf 2020 Mathematics Subject Classification}: 17B37
\end{abstract}

\maketitle 

\setcounter{tocdepth}{1}
\tableofcontents

\section{Introduction}
Shifted quantum affine algebras were introduced by Finkelberg and
Tsymbaliuk in their study of quantized $K$-theoretic Coulomb branches of
three-dimensional $\mathcal N=4$ supersymmetric quiver gauge theories
\cite{FT}. Since their introduction, shifted quantum affine algebras have appeared in a
number of related contexts, such as representation theory, Hall algebras,
supersymmetric gauge theory, and cluster algebra theory. In particular, various aspects of their representation theory, as well as their
connections with Hall algebras and cluster algebras, have been investigated in
\cite{Jordan-shifted,P-shifted,Hall-shifted,JEB-shifted,cluster-shifted}.

Shifted quantum affine algebras form a family of variants of quantum loop algebras in
Drinfeld's current realization \cite{FT}. More precisely, if $\g$ is a
finite-dimensional simple Lie algebra and $\mu$ is an integral coweight,
the shifted quantum affine algebra $\U_q^\mu(\hat{\g})$ is obtained
from the Drinfeld presentation of the quantum loop algebra
$\U_q(L\g)$ by allowing the Drinfeld currents to have prescribed
orders of zero or pole determined by the integers
$\alpha_i(\mu)$.  Thus the shift changes the asymptotic
behavior of the Drinfeld currents $\phi^\pm_i(z)$ and leads to a representation theory
which is close to, but genuinely different from, the usual representation
theory of quantum affine algebras \cite{Hernandezshfted}.

The representation theory of shifted quantum affine algebras was
developed by Hernandez in \cite{Hernandezshfted}.  In that work, the
category $\mo_\mu$ of representations of $\U_q^\mu(\hg)$ were introduced, simple objects were classified by
rational $\ell$-weights of prescribed degrees, and a fusion product was
constructed by means of a deformed Drinfeld coproduct.  This gives a
ring structure on the Grothendieck group of the shifted category
$\mo^{sh}
=
\bigoplus_{\mu}\mo_\mu$ .

The purpose of the present paper is to develop representations
theory for shifted twisted quantum affine algebras.  Let $\g$ be a simply-laced
simple Lie algebra of type $A_{2n}$ ($n \ge 2$), $A_{2n-1}$ $(n\ge 3)$, $D_{n+1}$ ($n \ge 2$), or $E_6$, and let $\sigma$ be a
non-trivial Dynkin diagram automorphism of order $M$. We say that $\g$ is of type $(X_m, M)$ if $\sigma$ has order $M$. The corresponding
twisted affine Lie algebra is denoted by $\hgs$.  The twisted quantum
affine algebra $\U_q(\hgs)$ was introduced independently by Drinfeld
and Jimbo \cite{Drinfeld,Jimbo}.  Its finite-dimensional simple
representations were classified by Chari and Pressley in terms of
Drinfeld polynomials \cite{chari-twisted-98}.  The Drinfeld new realization
of twisted quantum affine algebras was studied in
\cite{Drinfeld,I-Damiani,Damiani-injective,jing1998drinfeld,jing1999vertex}.
The theory of $q$-characters and Kirillov--Reshetikhin modules in the
twisted case was developed by Hernandez \cite{Hernandeztwisted}, and
the category $\mo_{\bs}$ for twisted quantum affine Borel algebras was studied
by Wang \cite{wang2023QQ}.  These works show that many structures from
the untwisted theory have twisted counterparts, although the diagram
automorphism produces additional features in the current relations,
the Serre relations, and the form of the $\ell$-weights.

In this paper, we introduce shifted twisted quantum
affine algebras and study their representation theory. Denote $I_0=\{1,\ldots,n\}$ when $\g$ is of type $(A_{2n},2)$, $(A_{2n-1},2)$, or $(D_{n+1},2)$, where $2$ is the order of $\sigma$, $I_0=\{1,2\}$ when $\g$ is of type $(D_4,3)$, $I_0 = \{1,2,3,4\}$ when $\g$ is of type $(E_6,2)$.  Let $\Lambda'=\bigoplus_{i\in I_0}\ZZ\omega_i^\vee$ be the coweight lattice of the finite type Lie algebra $\gs$ associated
with the twisted affine type.  For two coweights
$\mu_+,\mu_-\in \Lambda'$ satisfying the natural divisibility condition
\[
M\mid \alpha_i(\mu_\pm)
\qquad
\text{whenever } \sigma(i)=i,
\]
we define an algebra $\U_q^{\mu_+,\mu_-}(\hgs)$ by a shifted version of the Drinfeld current presentation of the
twisted quantum loop algebra. Its generators are the modes $x_{i,r}^\pm$, $\phi_{i,r}^+$, and $\phi_{i,r}^-$, where $i\in I_0$, $r\in\mathbb{Z}$ for $x_{i,r}^\pm$, $r\ge -\alpha_i(\mu_+)$ for $\phi_{i,r}^+$, and $r\le \alpha_i(\mu_-)$ for $\phi_{i,r}^-$. The defining relations are the twisted current relations expressed in terms of these modes; see Definition \ref{basic-definition}. In particular, the
ordinary twisted quantum loop algebra is recovered from the case
$\mu_+=\mu_-=0$ after identifying $\phi^+_{i,0}\phi^-_{i,0}$ with $1$ for $i\in I_0$.

As in the shifted untwisted case, the algebra
$\U_q^{\mu_+,\mu_-}(\hgs)$ depends, up to isomorphism, only on the sum $\mu=\mu_+ + \mu_-$; see Theorem \ref{up-to-iso}.
We therefore write
$\U_q^\mu(\hgs)
=
\U_q^{0,\mu}(\hgs)$
throughout the paper.  We prove a triangular decomposition for these
algebras, using the triangular decomposition of twisted quantum
affinizations from \cite{chen2023twisted}; see Proposition \ref{triangular}.  We also show that, for 
anti-dominant shift $\mu$, the algebra $\U_q^\mu(A_{2}^{(2)})$ contains a natural
subalgebra isomorphic to the twisted quantum affine Borel algebra
$\U_q(\bs)$; see Proposition \ref{borel-injective}.

Our first main result concerns the category $\mo_\mu$ of representations
of $\U_q^\mu(\hgs)$.  This category is defined by a weight-space
decomposition, finite-dimensionality of weight spaces, and a boundedness
condition with respect to the natural partial order on the weight
lattice.  
A key technical point is the rationality of the Drinfeld currents $\phi^\pm_i(z)$.  For
a module in $\mo_\mu$, we prove that on every weight space the currents
$\phi_i^+(z)$ and $\phi_i^-(z)$ are the expansions of the same rational
operator-valued function.  Moreover, this rational function has degree $\alpha_i(\mu)$; see Proposition \ref{key-proposition}.

Using this rationality result, we classify the simple objects of
$\mo_\mu$.  Let $\mathfrak r_\mu$ be the set of rational $\ell$-weights $\mathbf{\Psi}
=
(\Psi_i(z))_{i\in I_0}$
such that
$\deg \Psi_i(z)=\alpha_i(\mu), 
i\in I_0$.
We prove that the simple modules in $\mo_\mu$ are parametrized by
$\mathfrak r_\mu$; see Theorem \ref{rational-theorem}. Thus every simple object is a highest
$\ell$-weight module $L(\mathbf{\Psi})$ with rational highest
$\ell$-weight of degree prescribed by the shift, and conversely every
such rational $\ell$-weight gives a simple object in $\mo_\mu$.  This is
the twisted analogue of Hernandez's classification theorem for shifted
quantum affine algebras.

We then consider the direct sum category $\mo^{sh}
=\bigoplus_{\mu\in\Lambda}\mo_\mu$, where $\Lambda$ is a sublattice of $\Lambda^\prime$. Following the strategy of \cite{He2,Hernandezshfted}, we construct a
fusion product on $\mo^{sh}$.  For
$V_1\in \mo_{\mu_1}$ and $V_2\in \mo_{\mu_2}$, we define a
deformed Drinfeld coproduct
\[
\Delta_u:
\U_q^{\mu_1+\mu_2}(\hgs)
\longrightarrow
\bigl(\U_q^{\mu_1}(\hgs)\otimes
\U_q^{\mu_2}(\hgs)\bigr)((u)).
\]
We prove that this is an algebra homomorphism; see Proposition \ref{Delta_u}.  Then, for highest
$\ell$-weight modules $V_1$ and $V_2$, we show that the space of
rational Laurent series $(V_1\otimes V_2)(u)$ is stable under this
action and is cyclic over $\CC(u)$, generated by the tensor product of
the two highest $\ell$-weight vectors.  After choosing an
$\mathcal A$-form, where $\mathcal A$ is the ring of rational functions
regular at $u=1$, we specialize at $u=1$ and obtain a module $V_1*V_2\in \mo_{\mu_1+\mu_2}$, called the fusion product. This fusion product is compatible with $q$-characters.  More precisely,
we prove $\chi_q(V_1*V_2)
=\chi_q(V_1)\chi_q(V_2)$; see Theorem \ref{q-character-compality}. We also prove that every simple object of
$\mo^{sh}$ is a subquotient of a fusion product of prefundamental
representations and a constant representation; see Corollary \ref{subquotient}.  

We also study finite-dimensional representations of shifted twisted 
quantum affine algebras.  For the ordinary twisted quantum loop algebra,
finite-dimensional simple modules are classified by Drinfeld polynomials
\cite{chari-twisted-98,Hernandeztwisted}.  In the shifted setting, the
degree of the highest $\ell$-weight is no longer zero, and the existence
of finite-dimensional representations imposes a condition on the shift.
We prove that, except for the type
$A_{2n}^{(2)}$, the algebra $\U_q^\mu(\hgs)$ admits non-zero
finite-dimensional representations only when $\mu$ is dominant; see Theorem \ref{finite-dimensional-1}.   In
that case, the simple finite-dimensional modules are precisely the
modules $L(\mathbf{\Psi})$ whose highest $\ell$-weights are dominant in
the shifted twisted sense.  The type $A_{2n}^{(2)}$ requires a separate
argument, and we prove the corresponding classification in the positive
subcategory considered in this paper; see Theorem \ref{finite-dimensional-2}.

Finally, we construct restriction representations relating representations of twisted quantum affine Borel algebras to representations of shifted twisted quantum affine algebras. The restriction
construction is based on the following idea.  Given a representation in
the twisted Borel category, one considers the subspace on which the
rational Drinfeld currents $\phi^+_i(z)$ have degree at most $\alpha_i(\mu)$, and then
passes to the quotient by the subspace of smaller degree.  On this
quotient, one defines new operators
$\widetilde{x}_{i,m}^{\pm},
\phi_i^\pm(z)$
by taking the appropriate rational continuations of the Drinfeld
currents. We verify the defining relations of
$\U_q^\mu(\hgs)$ by using truncated versions of the twisted current and
Serre relations. As an application, for a simple finite-dimensional
$\U_q^\mu(\hgs)$-module $L(\mathbf{\Psi})$, we obtain the formula
\[
\chi_q(L^{\bs}(\mathbf{\Psi}))
=
\chi_q(L(\mathbf{\Psi}))
\prod_{i\in I_0}\chi_i^{\alpha_i(\mu)/\iota _i},
\]
where $\chi_i$ is the $q$-character factor associated with the positive
prefundamental representation $L^{\bs}(\mathbf{\Psi}_{i,1})$ of the twisted quantum affine Borel algebra $\U_q(\bs)$; see Theorem \ref{q-character-shifted-borel}.

The paper is organized as follows.  In Section~2, we recall the
Drinfeld new realization of twisted quantum affine algebras, their Borel
subalgebras, the categories $\mo_{\hgs}$ of representations of $\U_q(\lgs)$ and $\mo_{\bs}$ of representations of $\U_q(\bs)$, and several
facts on $q$-characters and Drinfeld coproducts which will be used
later.  In Section~3, we define the shifted twisted  quantum affine
algebras $\U_q^{\mu_+,\mu_-}(\hgs)$, prove their triangular
decomposition, show that they depend only on
$\mu=\mu_++\mu_-$. In Section~4, we introduce the categories
$\mo_\mu$ and classify their simple objects by rational $\ell$-weights.
In Section~5, we construct the deformed Drinfeld coproduct and the fusion
product. In Section~6, we
classify finite-dimensional simple representations of $\U_q^\mu(\hgs)$.  In Section~7, we
study  restriction representations and establish a $q$-characters formula for simple finite-dimensional representations of  shifted twisted quantum affine algebras in terms of the $q$-characters of the corresponding simple representations of the twisted quantum affine Borel algebra $\U_q(\bs)$.

\section*{Acknowledgments}
The authors are grateful to Alexandr Garbali, David Hernandez, Evgeny Mukhin, Alexander Tsymbaliuk, and Keyu Wang for helpful discussions and valuable comments. This work was supported by the National Natural Science Foundation of China (Nos. 12171213, 12571018). JRL was supported by the Austrian Science Fund (FWF), PAT 9039323, Grant-DOI 10.55776/PAT9039323, and the National Natural Science Foundation of China (No. 12471023).

\section{Twisted Quantum Affine Algebras}

This section fixes the notations and recalls the main facts on twisted quantum
affine algebras that will be used throughout the paper. We review the Drinfeld
realization of $\mathcal U_q(\mathcal L\gs)$, the corresponding Borel
subalgebra, the categories $\mathcal O_{\hgs}$ and $\mathcal O_{\bs}$, rational
$\ell$-weights, and the Drinfeld coproduct. These preliminaries provide the
foundation for the definition and representation theory of shifted twisted
quantum affine algebras developed in the following sections.
	
\subsection{Twisted affine Lie algebras} \label{subsec:Twisted affine Lie algebras}
We first fix some notations. An $n\times n$ matrix $C=(C_{i,j})_{1\leq i,j\leq n}$ is called a generalized Cartan matrix if $C_{i,i}=2$, $C_{i,j}\in\ZZ_{\leq 0}$ for $i\neq j$, and $C_{i,j}=0$ if and only if $C_{j,i}=0$. We shall always assume that $C$ is symmetrizable, namely that there exists a diagonal matrix $D=\operatorname{diag}(d_1,\ldots,d_n)$ with positive integer entries such that $DC$ is symmetric.

Let $q\in\CC^*$ be not a root of unity. For $n\in\ZZ$, $r\in\ZZ_{\geq 0}$, and integers $m\geq m'$, we use the standard notation
\[
[n]_q=\frac{q^n-q^{-n}}{q-q^{-1}},\qquad
[r]_q!=\prod_{s=1}^r[s]_q,\qquad
\begin{bmatrix}
	m\\
	m'
\end{bmatrix}_q
=
\frac{[m]_q!}{[m-m']_q![m']_q!}.
\]
We also use the convention $[0]_q!=1$.

Let $\g$ be a finite-dimensional complex simple Lie algebra with Cartan matrix $C=(C_{i,j})_{i,j\in I}$. Let $\sigma:I\to I$ be a diagram automorphism, that is, a bijection satisfying $C_{i,j}=C_{\sigma(i),\sigma(j)}$ for all $i,j\in I$. We denote by the same symbol $\sigma$ the induced Lie algebra automorphism of $\g$, and let $M$ be its order. Throughout this paper we assume that $M>1$. Hence $M=2$ or $3$, and $\g$ is of type $(A_n,2)$ $(n\geq 2,n\ne 3)$, $(D_{n+1},2)$ $(n\geq 2)$, $(D_4,3)$, or $(E_6,2)$. The Dynkin diagrams of $\g$, together with the labeling of the nodes, are given in Figure \ref{fig:simple-Lie-algebra-dynkin}.
\begin{figure}[htp]
    \centering
    \begin{tikzpicture}[
    dynkinnode/.style={circle, draw, inner sep=0pt, minimum size=4pt},
    labeltext/.style={font=\small},
    every label/.append style={font=\scriptsize}
]

    \node[labeltext] at (-2, 0) {$A_n$ $(n\ge2,n\ne3)$};
    
    \node[dynkinnode, label=below:{$1$}]   (A1)  at (0,0) {};
    \node[dynkinnode, label=below:{$2$}]   (A2)  at (0.9,0) {};
    \node[dynkinnode, label=below:{$3$}]   (A3)  at (1.8,0) {};
    \node (Adots) at (2.7,0) {$\cdots$};
    \node[dynkinnode, label=below:{$n-1$}] (An1) at (3.6,0) {};
    \node[dynkinnode, label=below:{$n$}]   (An)  at (4.5,0) {};

    \draw (A1) -- (A2) -- (A3) -- (Adots) -- (An1) -- (An);


    \begin{scope}[yshift=-1.5cm]
        \node[labeltext] at (-2, 0) {$D_{n+1}$ $(n\ge2)$};
        
        \node[dynkinnode, label=below:{$1$}]   (D1)  at (0,0) {};
        \node[dynkinnode, label=below:{$2$}]   (D2)  at (0.9,0) {};
        \node (Ddots) at (1.8,0) {$\cdots$};
        \node[dynkinnode, label=below:{$n-2$}] (Dn2) at (2.7,0) {};
        \node[dynkinnode, label=below:{$n-1$}] (Dn1) at (3.6,0) {};
        
        \node[dynkinnode, label=right:{$n$}]   (Dn)   at (4.4, 0.45) {};
        \node[dynkinnode, label=right:{$n+1$}] (Dnp1) at (4.4,-0.45) {};

        \draw (D1) -- (D2) -- (Ddots) -- (Dn2) -- (Dn1);
        \draw (Dn1) -- (Dn);
        \draw (Dn1) -- (Dnp1);

    \end{scope}

    \begin{scope}[yshift=-3.4cm]
        \node[labeltext] at (-2, 0) {$E_6$};
        
        \node[dynkinnode, label=below:{$1$}] (E1) at (0,0) {};
        \node[dynkinnode, label=below:{$2$}] (E2) at (0.9,0) {};
        \node[dynkinnode, label=below:{$3$}] (E3) at (1.8,0) {};
        
        \node[dynkinnode, label=right:{$4$}] (E4) at (1.8,0.9) {};
        
        \node[dynkinnode, label=below:{$5$}] (E5) at (2.7,0) {};
        \node[dynkinnode, label=below:{$6$}] (E6) at (3.6,0) {};

        \draw (E1) -- (E2) -- (E3) -- (E5) -- (E6);
        \draw (E3) -- (E4); 

    \end{scope}
\end{tikzpicture}
\caption{Dynkin diagrams for $\g$.}
\label{fig:simple-Lie-algebra-dynkin}
\end{figure}
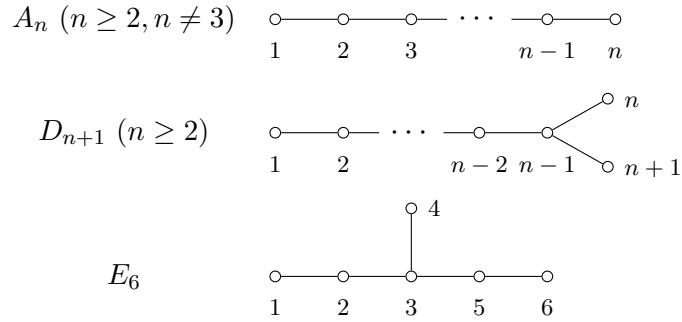

We choose, once and for all, a representative in each $\sigma$-orbit of $I$.
More precisely, with respect to the ordering of the vertices of the
Dynkin diagram, we choose the representative $i$ such that $\sigma^k(i)\geq i$
for all $k$. We denote by $I_0$ the set of these representatives. 
If $\g$ is of type $(A_{2n-1},2)$, $(A_{2n},2)$, or $(D_{n+1},2)$, then $I_0=\{1,\ldots,n\}$. If $\g$ is of type $(D_4,3)$, then $I_0=\{1,2\}$. If $\g$ is of type $(E_6,2)$, then $I_0=\{1,2,3,4\}$. The pair $(\g,\sigma)$ determines a twisted affine Lie algebra $\hgs$. It is the
Kac--Moody algebra with Cartan matrix $(C^\sigma_{i,j})_{i,j\in\hat{I}_\sigma}$,
where $\hat{I}_\sigma=I_0\sqcup\{0\}$ and $0$ denotes the additional affine node; see \cite{Kac}. The Dynkin diagrams of $\hgs$, together with the labeling of the nodes, are given in Figure \ref{fig:twisted-affine-dynkin}. 
\begin{figure}[htbp]
\centering
\begin{tikzpicture}[
  scale=.78, transform shape,
  x=1cm,y=1cm,
  line cap=round,
  line join=round,
  plain/.style={line width=.42pt},
  fnode/.style={circle,draw=black,fill=white,inner sep=0pt,minimum size=5.0pt},
  onode/.style={circle,draw=black,fill=white,inner sep=0pt,minimum size=5.0pt},
  every label/.style={font=\scriptsize,inner sep=.8pt},
  leftname/.style={anchor=east,font=\small}
]

\newcommand{\rmark}[1]{%
  \draw[plain] ($#1+(-.11,.16)$) -- ($#1$) -- ($#1+(-.11,-.16)$);
}
\newcommand{\lmark}[1]{%
  \draw[plain] ($#1+(.11,.16)$) -- ($#1$) -- ($#1+(.11,-.16)$);
}

\newcommand{\tworight}[2]{%
  \draw[plain] ($(#1)+(0,.045)$) -- ($(#2)+(0,.045)$);
  \draw[plain] ($(#1)+(0,-.045)$) -- ($(#2)+(0,-.045)$);
  \rmark{(#1)!.55!(#2)}%
}
\newcommand{\twoleft}[2]{%
  \draw[plain] ($(#1)+(0,.045)$) -- ($(#2)+(0,.045)$);
  \draw[plain] ($(#1)+(0,-.045)$) -- ($(#2)+(0,-.045)$);
  \lmark{(#1)!.55!(#2)}%
}
\newcommand{\threeleft}[2]{%
  \draw[plain] ($(#1)+(0,.07)$) -- ($(#2)+(0,.07)$);
  \draw[plain] (#1) -- (#2);
  \draw[plain] ($(#1)+(0,-.07)$) -- ($(#2)+(0,-.07)$);
  \lmark{(#1)!.58!(#2)}%
}
\newcommand{\fourright}[2]{%
  \draw[plain] ($(#1)+(0,.105)$) -- ($(#2)+(0,.105)$);
  \draw[plain] ($(#1)+(0,.035)$) -- ($(#2)+(0,.035)$);
  \draw[plain] ($(#1)+(0,-.035)$) -- ($(#2)+(0,-.035)$);
  \draw[plain] ($(#1)+(0,-.105)$) -- ($(#2)+(0,-.105)$);
  \rmark{(#1)!.55!(#2)}%
}

\node[leftname] at (1.55,0) {$A^{(2)}_2$};
\begin{scope}[shift={(9.55,0)}]
  \coordinate (a20) at (0,0);
  \coordinate (a21) at (1.15,0);
  \fourright{a20}{a21}
  \node[onode,label=below:{$0$}] at (a20) {};
  \node[fnode,label=below:{$1$}] at (a21) {};
  
\end{scope}

\node[leftname] at (1.55,-1.25) {$A^{(2)}_{2n}$ $(n\ge 2)$};
\begin{scope}[shift={(3.35,-1.25)}]
  \coordinate (b0)  at (0,0);
  \coordinate (b1)  at (1.15,0);
  \coordinate (b2)  at (2.30,0);
  \coordinate (b3)  at (3.45,0);
  \coordinate (bn2) at (5.75,0);
  \coordinate (bn1) at (6.90,0);
  \coordinate (bn)  at (8.05,0);
  \tworight{b0}{b1}
  \draw[plain] (b1) -- (b2) -- (b3);
  \draw[plain] (b3) -- ++(.65,0);
  \draw[plain,densely dotted] ($(b3)+(.65,0)$) -- ($(bn2)+(-.65,0)$);
  \draw[plain] ($(bn2)+(-.65,0)$) -- (bn2) -- (bn1);
  \tworight{bn1}{bn}
  \node[onode,label=below:{$0$}] at (b0) {};
  \node[fnode,label=below:{$1$}] at (b1) {};
  \node[fnode,label=below:{$2$}] at (b2) {};
  \node[fnode,label=below:{$3$}] at (b3) {};
  \node[fnode,label=below:{$n-2$}] at (bn2) {};
  \node[fnode,label=below:{$n-1$}] at (bn1) {};
  \node[fnode,label=below:{$n$}] at (bn) {};
  
\end{scope}

\node[leftname] at (1.55,-3.15) {$A^{(2)}_{2n-1}$ $(n\ge 3)$};
\begin{scope}[shift={(3.80,-3.15)}]
  \coordinate (c0)  at (0,.72);
  \coordinate (c1)  at (0,-.72);
  \coordinate (c2)  at (.75,0);
  \coordinate (c3)  at (1.90,0);
  \coordinate (c4)  at (3.05,0);
  \coordinate (cn2) at (5.35,0);
  \coordinate (cn1) at (6.50,0);
  \coordinate (cn)  at (7.65,0);
  \draw[plain] (c0) -- (c2) -- (c1);
  \draw[plain] (c2) -- (c3) -- (c4);
  \draw[plain] (c4) -- ++(.65,0);
  \draw[plain,densely dotted] ($(c4)+(.65,0)$) -- ($(cn2)+(-.65,0)$);
  \draw[plain] ($(cn2)+(-.65,0)$) -- (cn2) -- (cn1);
  \twoleft{cn1}{cn}
  \node[onode,label=left:{$0$}] at (c0) {};
  \node[fnode,label=left:{$1$}] at (c1) {};
  \node[fnode,label=left:{$2$}] at (c2) {};
  \node[fnode,label=below:{$3$}] at (c3) {};
  \node[fnode,label=below:{$4$}] at (c4) {};
  \node[fnode,label=below:{$n-2$}] at (cn2) {};
  \node[fnode,label=below:{$n-1$}] at (cn1) {};
  \node[fnode,label=below:{$n$}] at (cn) {};
  
\end{scope}

\node[leftname] at (1.55,-4.90) {$D^{(2)}_{n+1}$ $(n\ge 2)$};
\begin{scope}[shift={(2.75,-4.90)}]
  \coordinate (d0)  at (0,0);
  \coordinate (d1)  at (1.15,0);
  \coordinate (d2)  at (2.30,0);
  \coordinate (d3)  at (3.45,0);
  \coordinate (dn3) at (5.75,0);
  \coordinate (dn2) at (6.90,0);
  \coordinate (dn1) at (8.05,0);
  \coordinate (dn)  at (9.20,0);
  \twoleft{d0}{d1}
  \draw[plain] (d1) -- (d2) -- (d3);
  \draw[plain] (d3) -- ++(.65,0);
  \draw[plain,densely dotted] ($(d3)+(.65,0)$) -- ($(dn3)+(-.65,0)$);
  \draw[plain] ($(dn3)+(-.65,0)$) -- (dn3) -- (dn2) -- (dn1);
  \tworight{dn1}{dn}
  \node[onode,label=below:{$0$}] at (d0) {};
  \node[fnode,label=below:{$1$}] at (d1) {};
  \node[fnode,label=below:{$2$}] at (d2) {};
  \node[fnode,label=below:{$3$}] at (d3) {};
  \node[fnode,label=below:{$n-3$}] at (dn3) {};
  \node[fnode,label=below:{$n-2$}] at (dn2) {};
  \node[fnode,label=below:{$n-1$}] at (dn1) {};
  \node[fnode,label=below:{$n$}] at (dn) {};
  
\end{scope}

\node[leftname] at (1.55,-6.15) {$E^{(2)}_6$};
\begin{scope}[shift={(6.65,-6.15)}]
  \coordinate (e0) at (0,0);
  \coordinate (e1) at (1.15,0);
  \coordinate (e2) at (2.30,0);
  \coordinate (e3) at (3.45,0);
  \coordinate (e4) at (4.60,0);
  \draw[plain] (e0) -- (e1) -- (e2);
  \twoleft{e2}{e3}
  \draw[plain] (e3) -- (e4);
  \node[onode,label=below:{$0$}] at (e0) {};
  \node[fnode,label=below:{$1$}] at (e1) {};
  \node[fnode,label=below:{$2$}] at (e2) {};
  \node[fnode,label=below:{$3$}] at (e3) {};
  \node[fnode,label=below:{$4$}] at (e4) {};
  
\end{scope}

\node[leftname] at (1.55,-7.35) {$D^{(3)}_4$};
\begin{scope}[shift={(7.95,-7.35)}]
  \coordinate (f0) at (0,0);
  \coordinate (f1) at (1.15,0);
  \coordinate (f2) at (2.30,0);
  \draw[plain] (f0) -- (f1);
  \threeleft{f1}{f2}
  \node[onode,label=below:{$0$}] at (f0) {};
  \node[fnode,label=below:{$1$}] at (f1) {};
  \node[fnode,label=below:{$2$}] at (f2) {};
  
\end{scope}

\end{tikzpicture}
\caption{Dynkin diagrams for $\hgs$.}
\label{fig:twisted-affine-dynkin}
\end{figure}


We denote by $\gs$ the finite-dimensional simple Lie algebra whose Cartan matrix is $(C^\sigma_{i,j})_{i,j\in I_0}$. Except in type $A_2^{(2)}$, this finite-type Cartan matrix is not simply laced. The finite type $\gs$ corresponding to the twisted affine type $\hgs$ is given by the following table:
\begin{align} \label{eq:table of affine Lie algebra}
    \begin{array}{c|c|c|c|c|c|c}
	\hgs
	&
	A_{2}^{(2)}
	&
	A_{2n}^{(2)}
	&
	A_{2n-1}^{(2)}
	&
	D_{n+1}^{(2)}
	&
	E_{6}^{(2)}
	&
	D_{4}^{(3)}
	\\ \hline
	\gs
	&
	A_1
	&
	B_n
	&
	C_n
	&
	B_n
	&
	F_4
	&
	G_2
\end{array}
\qquad (n\geq 2).
\end{align}

The Cartan matrix $(C^\sigma_{i,j})_{i,j\in\hat I_\sigma}$ is symmetrizable. We fix the diagonal symmetrizing matrix $D=\operatorname{diag}(d_i)_{i\in\hat I_\sigma}$ as in \cite[Section~2.4]{Hernandeztwisted}, and put $q_i=q^{d_i}$ for $i\in\hat I_\sigma$. Explicitly, $d_i$ is given by 
\begin{equation}
\label{eq:my_d_i} 
\begin{aligned}
    (2,\frac{1}{2}) &\quad \text{for type } A_2^{(2)}, \\
    (2,1,\ldots,1,\frac{1}{2}) &\quad \text{for type } A_{2n}^{(2)}\ (n \ge 2), \\
    (1, \ldots, 1, 2) &\quad \text{for type } A_{2n-1}^{(2)}\ (n \ge 3), \\
    (1, 2, \ldots, 2, 1) &\quad \text{for type } D_{n+1}^{(2)}\ (n \ge 2), \\
    (1, 1, 1, 2, 2) &\quad \text{for type } E_6^{(2)}, \\
    (1, 1, 3) &\quad \text{for type } D_4^{(3)}.
\end{aligned}
\end{equation}

\subsection{Twisted quantum affine algebras}
\begin{definition}
	The twisted quantum affine algebra $\mathcal U_q(\hgs)$ is the associative
	$\CC$-algebra generated by $e_i^\pm$ and $k_i^{\pm1}$, $i\in\hat I_\sigma$,
	subject to the following relations:
	\begin{align*}
		&k_i k_i^{-1}=k_i^{-1}k_i=1,\qquad k_i k_j=k_j k_i,\\
		&k_i e_j^\pm=q_i^{\pm C_{i,j}^\sigma}e_j^\pm k_i,\\
		&[e_i^+,e_j^-]
		=
		\delta_{i,j}\frac{k_i-k_i^{-1}}{q_i-q_i^{-1}},\\
		&\sum_{r=0}^{1-C_{i,j}^\sigma}
		(-1)^r
		(e_i^\pm)^{(1-C_{i,j}^\sigma-r)}
		e_j^\pm
		(e_i^\pm)^{(r)}
		=0
		\qquad (i\neq j).
	\end{align*}
	Here $(e_i^\pm)^{(r)}=(e_i^\pm)^r/[r]_{q_i}!$ denotes the divided power.
\end{definition}

The algebra $\mathcal U_q(\hgs)$ is a Hopf algebra with coproduct, antipode and
counit given by
\begin{align*}
	&\Delta(k_i^{\pm1})=k_i^{\pm1}\otimes k_i^{\pm1},\\
	&\Delta(e_i^+)=e_i^+\otimes 1+k_i\otimes e_i^+,
	\qquad
	\Delta(e_i^-)=e_i^-\otimes k_i^{-1}+1\otimes e_i^-,\\
	&S(k_i^{\pm1})=k_i^{\mp1},
	\qquad
	S(e_i^+)=-k_i^{-1}e_i^+,
	\qquad
	S(e_i^-)=-e_i^-k_i,\\
	&\varepsilon(k_i^{\pm1})=1,
	\qquad
	\varepsilon(e_i^\pm)=0.
\end{align*}
Thus
\[
S^{-1}(k_i^{\pm1})=k_i^{\mp1},
\qquad
S^{-1}(e_i^+)=-e_i^+k_i^{-1},
\qquad
S^{-1}(e_i^-)=-k_i e_i^- .
\]

Let $(a_i)_{i\in\hat I_\sigma}$ be the unique family of positive integers,
normalized by $a_0=1$, such that
\[
\sum_{i\in\hat I_\sigma} a_i d_i C_{i,j}^\sigma=0
\qquad
(j\in\hat I_\sigma).
\]
Then the element $c=\prod_{i\in\hat I_\sigma}k_i^{a_i}$ is central in
$\mathcal U_q(\hgs)$.

\begin{definition}
	The twisted quantum loop algebra $\mathcal U_q(\mathcal L\gs)$ is the quotient
	of $\mathcal U_q(\hgs)$ by the ideal generated by $c-1$.
\end{definition}

\begin{remark}\label{grading}
	The algebra $\mathcal U_q(\mathcal L\gs)$ has a $\ZZ$-grading determined by
	$\deg(e_i^+)=1$, $\deg(e_i^-)=-1$ and $\deg(k_i^{\pm1})=0$ for
	$i\in\hat I_\sigma$.
\end{remark}

The twisted quantum loop algebra $\mathcal{U}_q(\mathcal{L}\gs)$  has a  Drinfeld presentation. The generators appearing in this realization will be
called the Drinfeld generators. We set
\[
I_{\ZZ}=\{(i,r)\in I_0\times\ZZ :  d_i\mid r\},
\]
where $d_i$ is in (\ref{eq:my_d_i}).

We shall use the following notation. For $i,j\in I_0$, define
$d_{i,j}\in\QQ$ and $P_{i,j}^{\pm}(u_1,u_2)\in\QQ[u_1,u_2]$ as follows:
\begin{itemize}
    \item if $C_{i,\sigma(i)}=2$, then $d_{i,j}=1/2$ and
    $P_{i,j}^{\pm}(u_1,u_2)=1$;

    \item if $C_{i,\sigma(i)}=0$ and $\sigma(j)\neq j$, then
    $d_{i,j}=1/(2M)$ and $P_{i,j}^{\pm}(u_1,u_2)=1$;

    \item if $C_{i,\sigma(i)}=0$ and $\sigma(j)=j$, then
    $d_{i,j}=1/2$ and
    \[
    P_{i,j}^{\pm}(u_1,u_2)
    =
    \frac{u_1^M q^{\pm 2M}-u_2^M}{u_1q^{\pm2}-u_2};
    \]

    \item if $C_{i,\sigma(i)}=-1$, then $d_{i,j}=1/8$ and
    $P_{i,j}^{\pm}(u_1,u_2)=u_1q^{\pm1}+u_2$.
\end{itemize}

Finally, for $i\in I_0$, we define $N_i$ to be $M$ if $\sigma(i)=i$, and to be
$1$ otherwise. Thus, \(N_i = d_i\) for all affine types $\hgs$ in \eqref{eq:table of affine Lie algebra} and all \(i \in [n]\), except when $\hgs$ is of type \(A_{2n}^{(2)}\) and \(i = n\), where \(N_n = 1 \neq \tfrac{1}{2} = d_n\).
	
\begin{theorem}[{\cite{I-Damiani,Drinfeld,jing1998drinfeld,jing1999vertex}}]
	\label{thm:Drinfeld-presentation}
	The twisted quantum loop algebra $\mathcal U_q(\mathcal L\gs)$ admits the following Drinfeld presentation. It is generated by
	\[
	x_{i,r}^{\pm} \ (i\in I_0,\ r\in\ZZ),\qquad
	h_{i,m}\ (i\in I_0,\ m\in\ZZ\setminus\{0\}),\qquad
	k_i^{\pm1}\ (i\in I_0),
	\]
	subject to the following relations:
	\begin{align*}
		&x_{i,r}^{\pm}=0,
		\qquad (i,r)\in (I_0\times\ZZ)\setminus I_{\ZZ},\\
		&h_{i,m}=0,
		\qquad (i,m)\in (I_0\times\ZZ)\setminus I_{\ZZ},\\
		&[k_i,k_j]=[k_i,h_{j,m}]=[h_{i,m},h_{j,m'}]=0,\\
		&k_i x_{j,r}^{\pm}
		=
		q^{\pm\sum_{s=1}^{M}C_{i,\sigma^s(j)}}x_{j,r}^{\pm}k_i,\\
		&[h_{i,m},x_{j,r}^{\pm}]
		=
		\pm \frac{1}{m}
		\left(
		\sum_{s=1}^{M}
		\left[
		\frac{mC_{i,\sigma^s(j)}}{d_{i}}
		\right]_{q_{i}}
		\zeta^{ms}
		\right)
		x_{j,m+r}^{\pm},\\
		&[x_{i,r}^{+},x_{j,r'}^{-}]
		=
		\frac{\sum_{s=1}^{M}\delta_{\sigma^s(i),j}\zeta^{sr'}}{N_i}
		\frac{\phi_{i,r+r'}^{+}-\phi_{i,r+r'}^{-}}
		{q_{i}-q_{i}^{-1}} .
	\end{align*}
	Here $\zeta=e^{2\pi\rm i/M}$ and the elements $\phi_{i,\pm m}^{\pm}$ are defined by the generating series
	\[
	\phi_i^{\pm}(z)
	=
	\sum_{m\geq 0}\phi_{i,\pm m}^{\pm}z^{\pm m}
	=
	k_i^{\pm1}
	\exp\left(
	\pm(q_{i}-q_{i}^{-1})
	\sum_{m\geq 1}h_{i,\pm m}z^{\pm m}
	\right).
	\]
	
	Set $x_i^{\pm}(u)=\sum_{r\in\ZZ}x_{i,r}^{\pm}u^{-r}$. In terms of these currents, the remaining Drinfeld--Serre relations are as follows. For all $i,j\in I_0$,
	\begin{align*}
		&\prod_{s=1}^{M}
		\left(u_1-\zeta^s q^{\pm C_{i,\sigma^s(j)}}u_2\right)
		x_i^{\pm}(u_1)x_j^{\pm}(u_2)=
		\prod_{s=1}^{M}
		\left(u_1q^{\pm C_{i,\sigma^s(j)}}-\zeta^s u_2\right)
		x_j^{\pm}(u_2)x_i^{\pm}(u_1).
	\end{align*}
	If $C_{i,j}=-1$ and $\sigma(i)\neq j$, then
	\begin{align*}
		\operatorname{Sym}_{u_1,u_2}
		\Bigl\{
		P_{i,j}^{\pm}(u_1,u_2)
		\bigl(
		&x_j^{\pm}(v)x_i^{\pm}(u_1)x_i^{\pm}(u_2)\\
		&-(q^{2Md_{i,j}}+q^{-2Md_{i,j}})
		x_i^{\pm}(u_1)x_j^{\pm}(v)x_i^{\pm}(u_2)\\
		&+x_i^{\pm}(u_1)x_i^{\pm}(u_2)x_j^{\pm}(v)
		\bigr)
		\Bigr\}=0 .
	\end{align*}
	If $C_{i,\sigma(i)}=-1$, then
	\begin{align*}
		&\operatorname{Sym}_{u_1,u_2,u_3}
		\Bigl\{
		\bigl(
		q^{3/2}u_1^{\mp1}
		-(q^{1/2}+q^{-1/2})u_2^{\mp1}
		+q^{-3/2}u_3^{\mp1}
		\bigr)
		x_i^{\pm}(u_1)x_i^{\pm}(u_2)x_i^{\pm}(u_3)
		\Bigr\}=0,\\
		&\operatorname{Sym}_{u_1,u_2,u_3}
		\Bigl\{
		\bigl(
		q^{-3/2}u_1^{\pm1}
		-(q^{1/2}+q^{-1/2})u_2^{\pm1}
		+q^{3/2}u_3^{\pm1}
		\bigr)
		x_i^{\pm}(u_1)x_i^{\pm}(u_2)x_i^{\pm}(u_3)
		\Bigr\}=0 .
	\end{align*}
\end{theorem}

By \cite[Theorem~9.2]{Damiani-injective} and \cite[Theorem~4.2]{chen2023twisted}, one has the following triangular decomposition.

\begin{theorem}
	\label{thm:triangular-decomposition-loop}
	The multiplication map induces a vector-space isomorphism
	\[
	\mathcal U_q(\mathcal L\gs)^+
	\otimes
	\mathcal U_q(\mathcal L\gs)^0
	\otimes
	\mathcal U_q(\mathcal L\gs)^-
	\simeq
	\mathcal U_q(\mathcal L\gs),
	\]
	where $\mathcal U_q(\mathcal L\gs)^\pm$ is the subalgebra generated by the elements
	$x_{i,r}^{\pm}$, and $\mathcal U_q(\mathcal L\gs)^0$ is the subalgebra generated by the elements
	$h_{i,m}$ and $k_i^{\pm1}$.
\end{theorem}

\begin{remark}[{\cite[Remark 2.7]{wang2023QQ}}]\label{Drinfeld-coproduct-1}
 The twisted quantum loop algebra $\U_q(\lgs)$ admits the Drinfeld coproduct
 \[
 \Delta_D:\U_q(\lgs)\rightarrow \mathcal{U} _{\boldsymbol{q}}(\mathcal{L} \mathfrak{g} ^{\sigma})\widehat{\otimes }\mathcal{U} _{\boldsymbol{q}}(\mathcal{L} \mathfrak{g} ^{\sigma})
 \]
 defined in the Drinfeld  presentation of $\U_q(\lgs)$ via
 \begin{align*}
&\Delta _{_D}(x_{i}^{+}(z))=x_{i}^{+}(z)\otimes 1+\phi _{i}^{-}(z^{-1})\otimes x_{i}^{+}(z),\\
&\Delta _{_D}(x_{i}^{-}(z))=1\otimes x_{i}^{-}(z)+x_{i}^{-}(z)\otimes \phi _{i}^{+}(z^{-1}),\\
&\Delta _D(\phi _{i}^{\pm}(z))=\phi _{i}^{\pm}(z)\otimes \phi _{i}^{\pm}(z).
\end{align*}
\end{remark}
	
 
		
	

\subsection{Borel subalgebras}
\begin{definition}[{\cite{wang2023QQ}}]
The Borel subalgebra $\mathcal U_q(\bs)$ is the subalgebra of
$\mathcal U_q(\mathcal L\gs)$ generated by $e_i^+$ and $k_i^{\pm1}$,
$i\in\hat I_\sigma=I_\sigma\sqcup\{0\}$.
\end{definition}
    
The Borel subalgebra is the algebra with the above generators and the Weyl and Serre relations among them.
	
\begin{theorem}[{\cite[Lemma~4.21]{quantum-groups}}]
\label{thm:Borel-presentation}
The algebra $\mathcal U_q(\bs)$ is generated by $e_i^+$ and $k_i^{\pm1}$,
$i\in\hat I_\sigma$, subject to the relations
\begin{align*}
&k_i k_i^{-1}=k_i^{-1}k_i=1,\qquad k_i k_j=k_j k_i,\\
&k_i e_j^+=q_i^{C_{i,j}^\sigma}e_j^+k_i,\\
&\sum_{r=0}^{1-C_{i,j}^\sigma}
(-1)^r
(e_i^+)^{(1-C_{i,j}^\sigma-r)}
e_j^+
(e_i^+)^{(r)}
=0
\qquad (i\neq j).
\end{align*}
Here $(e_i^+)^{(r)}=(e_i^+)^r/[r]_{q_i}!$ denotes the divided power.
\end{theorem}
	
\begin{remark}
(i) \cite[Section 3]{wang2023QQ} The Borel subalgebra $\mathcal U_q(\bs)$ contains the Drinfeld generators
$x_{i,m}^+$, $k_i^{\pm1}$ and $x_{i,r}^-$ for $i\in I_0$, $m\geq 0$ and
$r>0$, subject to the usual divisibility condition on the modes. Hence it also
contains the elements $\phi_{i,m}^+$ for $i\in I_0$ and $m\geq 0$. However,
$\mathcal U_q(\bs)$ is not generated by these Drinfeld generators alone; in
particular, the Chevalley generator $e_0^+$ attached to the affine node
is not generated by them.

(ii) By \cite[Lemma~3.3]{wang2023QQ}, the Borel subalgebra
$\mathcal U_q(\bs)$ has a triangular decomposition. More precisely, the
multiplication map induces a vector-space isomorphism
\[
\mathcal U_q(\bs)^-
\otimes
\mathcal U_q(\bs)^0
\otimes
\mathcal U_q(\bs)^+
\simeq
\mathcal U_q(\bs),
\]
where
$\mathcal U_q(\bs)^\pm=\mathcal U_q(\bs)\cap
\mathcal U_q(\mathcal L\gs)^\pm$ and
$\mathcal U_q(\bs)^0=\mathcal U_q(\bs)\cap
\mathcal U_q(\mathcal L\gs)^0$.
\end{remark}

We will also use the negative Borel subalgebra $\U_q(\bs_-)$ generated by $e_i^-$, $k_i^{\pm 1}$, $i\in\hat I_\sigma=I_\sigma\sqcup\{0\}$. Similarly, we also have triangular decomposition of negative Borel subalgebra $\U_q(\bs_-)$ with $\mathcal U_q(\bs_-)^\pm=\mathcal U_q(\bs_-)\cap
\mathcal U_q(\mathcal L\gs)^\pm$ and
$\mathcal U_q(\bs_-)^0=\mathcal U_q(\bs_-)\cap
\mathcal U_q(\mathcal L\gs)^0$.

It is known that the universal $R$-matrix belongs to
$\mathcal U_q(\bs)\widehat{\otimes}\mathcal U_q(\mathfrak b_-^\sigma)$
\cite{Damianimatrixtwisted}. Moreover, it admits the following factorization
\cite[Remark~7.1.7]{Damianimatrixtwisted}:
\[
\mathcal R
=
\mathcal R_{<0}\mathcal R_{=0}\mathcal R_{>0}q^{-t_\infty}.
\]
Here
\begin{align*}
\mathcal R_{<0}
&=
\prod_{m\leq 0}
\exp_{\beta_m}
\bigl((q_{\beta_m}^{-1}-q_{\beta_m})
E_{\beta_m}\otimes F_{\beta_m}\bigr)
\in
\mathcal U_q(\bs)^+\widehat{\otimes}
\mathcal U_q(\mathfrak b_-^\sigma)^-,\\
\mathcal R_{=0}
&=
\prod_{r>0}\exp \widetilde C_r
\in
\mathcal U_q(\bs)^0\widehat{\otimes}
\mathcal U_q(\mathfrak b_-^\sigma)^0,\\
\mathcal R_{>0}
&=
\prod_{m>1}
\exp_{\beta_m}
\bigl((q_{\beta_m}^{-1}-q_{\beta_m})
E_{\beta_m}\otimes F_{\beta_m}\bigr)
\in
\mathcal U_q(\bs)^-\widehat{\otimes}
\mathcal U_q(\mathfrak b_-^\sigma)^+,
\end{align*}
where $\beta_m$ is defined in  \cite[Sections 2,\:3]{Damianimatrixtwisted}.

We shall use the following relation between the standard coproduct and the
Drinfeld coproduct.

\begin{lemma}[{\cite[Lemma~3.5]{wang2023QQ}}]
\label{lem:standard-Drinfeld-coproduct}
For $x\in \mathcal U_q(\lgs)$, the standard coproduct $\Delta$ and the
Drinfeld coproduct $\Delta_D$ are related by
\[
\Delta(x)
=
(\mathcal R_{<0}^{\lor})^{-1}
\Delta_D(x)
\mathcal R_{<0}^{\lor},
\]
where \(\tau : \mathcal{U}_q(\lgs) \otimes \mathcal{U}_q(\lgs) \to \mathcal{U}_q(\lgs) \otimes \mathcal{U}_q(\lgs)\) is the flip map given by \(\tau(a \otimes b)=b \otimes a\), and $\mathcal R_{<0}^{\lor}=\tau\mathcal R_{<0}$ is an invertible homogeneous
element of
\[
\mathcal U_q(\mathfrak b_-^\sigma)^-
\widehat{\otimes}
\mathcal U_q(\bs)^+
\]
of $\ZZ$-degree zero and with constant term $1\otimes 1$. Here the
$\ZZ$-grading is the one defined in Remark~\ref{grading}.
\end{lemma}

\subsection{\texorpdfstring{Category $\mathcal{O}_{\hgs}$ and $\mathcal{O}_{\bs}$}{Category CLgs and Obs}}
Let $\mathfrak t$ be the commutative subalgebra of $\mathcal U_q(\bs)$
generated by the elements $k_i^{\pm1}$, $i\in I_0$. Set
$\mathfrak t^*=(\CC^*)^{I_0}$, endowed with componentwise multiplication. For $\omega,\omega'\in\mathfrak t^*$, their product and inverse are given
by $(\omega\omega')(i)=\omega(i)\omega'(i)$, $
\omega^{-1}(i)=\omega(i)^{-1}\ (i\in I_0)$.
We regard an element $\omega\in\mathfrak t^*$ as a character of $\mathfrak t$
by setting $\omega(k_i)=\omega(i)$.

Let $P=\bigoplus_{i\in I_0}\ZZ\alpha_i$ be the root lattice of $\gs$. We define a group morphism
$P\to\mathfrak t^*$, $\alpha\mapsto \overline{\alpha}$, by setting
\[
\overline{\alpha_i}(j)=q_i^{C_{i,j}^\sigma}
\qquad (i,j\in I_0).
\]
This gives a partial order on $\mathfrak t^*$ as follows: for
$\omega,\omega'\in\mathfrak t^*$, we write $\omega'\leq \omega$ if
\[
\omega(\omega')^{-1}
\in
\left\langle \overline{\alpha_i}\mid i\in I_0\right\rangle_{\ZZ_{\geq0}},
\]
that is, if $\omega(\omega')^{-1}$ is a product of the elements
$\overline{\alpha_i}$, $i\in I_0$.

If $\omega(\omega')^{-1}
=
\overline{\alpha_{i_1}}\:\overline{\alpha_{i_2}}\cdots
\overline{\alpha_{i_s}}
$ $(i_1,\ldots,i_s\in I_0)$,
we define the height of $\omega(\omega')^{-1}$ to be $s$.

\begin{definition}
Let $V$ be a $\mathcal U_q(\mathcal L\gs)$-module. For
$\lambda=(\lambda(i))_{i\in I_0}\in\mathfrak t^*$, the weight space of $V$
of weight $\lambda$ is
\[
V_\lambda
=
\{v\in V\mid k_i v=\lambda(i)v \text{ for all } i\in I_0\}.
\]
If $V_\lambda\neq 0$, then $\lambda$ is called a weight of $V$. A nonzero
vector in $V_\lambda$ is called a vector of weight $\lambda$.
\end{definition}

Let $\mathfrak t^*_{\ell,\hgs}$ be the set of all pairs
$\mathbf\Psi=(\mathbf\Psi^+,\mathbf\Psi^-)$, where
$\mathbf\Psi^\pm=(\Psi_i^\pm(z))_{i\in I_0}$ and
\[
\Psi_i^\pm(z)
=
\sum_{m\geq 0}\Psi_{i,\pm m}^{\pm}z^{\pm m}
\in \CC[[z^{\pm1}]]
\qquad (i\in I_0),
\]
such that $\Psi_{i,0}^+\Psi_{i,0}^-=1$ for all $i\in I_0$.

\begin{definition}
Let $V$ be a $\mathcal U_q(\mathcal L\gs)$-module and let
$\mathbf\Psi\in\mathfrak t^*_{\ell,\hgs}$. The $\ell$-weight space of $V$
associated with $\mathbf\Psi$ is the simultaneous generalized eigenspace
\[
V_{\mathbf\Psi}
=
\left\{
v\in V\ \middle|\ 
\begin{array}{l}
\text{there exists } p\geq 1 \text{ such that, for all } i\in I_0
\text{ and } m\geq 0,\\
(\phi_{i,\pm m}^{\pm}-\Psi_{i,\pm m}^{\pm})^p v=0
\end{array}
\right\}.
\]
If $V_{\mathbf\Psi}\neq 0$, then $\mathbf\Psi$ is called an $\ell$-weight of
$V$. A nonzero vector in $V_{\mathbf\Psi}$ is called a vector of
$\ell$-weight $\mathbf\Psi$.
\end{definition}

We recall the notion of highest $\ell$-weight modules
\cite{chari-twisted-98,Hernandeztwisted}.

\begin{definition}
A $\mathcal U_q(\mathcal L\gs)$-module $V$ is called a highest
$\ell$-weight module if there exist a vector $v\in V$ and an element
$\mathbf\Psi=(\mathbf\Psi^+,\mathbf\Psi^-)\in\mathfrak t^*_{\ell,\hgs}$ such
that $V=\mathcal U_q(\mathcal L\gs)v$ and
\[
x_{i,r}^+v=0,
\qquad
\phi_{i,\pm m}^{\pm}v=\Psi_{i,\pm m}^{\pm}v,
\qquad
i\in I_0,\ r\in\ZZ,\ m\geq 0.
\]
In this case $V$ is said to have highest $\ell$-weight $\mathbf\Psi$, and
$v$ is called a highest $\ell$-weight vector.

Equivalently, the highest $\ell$-weight is encoded by the formal series
\[
\Psi_i^+(z)=\sum_{m\geq 0}\Psi_{i,m}^+z^m\in\CC[[z]],
\qquad
\Psi_i^-(z)=\sum_{m\geq 0}\Psi_{i,-m}^-z^{-m}\in\CC[[z^{-1}]]
\qquad
(i\in I_0).
\]
\end{definition}

For each $\mathbf{\Psi}\in\mathfrak t^*_{\ell,\hgs}$, there exists a unique
simple highest $\ell$-weight $\mathcal U_q(\mathcal L\gs)$-module with highest
$\ell$-weight $\mathbf{\Psi}$. We denote this module by $L(\mathbf{\Psi})$.

We now define the category $\mathcal O_{\hgs}$ of
$\mathcal U_q(\mathcal L\gs)$-modules. A $\mathcal U_q(\mathcal L\gs)$-module $V$ belongs to $\mathcal O_{\hgs}$ if the
following conditions are satisfied:
\begin{enumerate}
\item $V$ is $\mathfrak t$-diagonalizable;
\item $\dim V_\lambda<\infty$ for all $\lambda\in\mathfrak t^*$;
\item there exist $\lambda_1,\ldots,\lambda_N\in\mathfrak t^*$ such that the
weights of $V$ are contained in $\bigcup_{s=1}^N D(\lambda_s)$, where
\[
D(\lambda_s)=\{\omega\in\mathfrak t^*\mid \omega\leq \lambda_s\}.
\]
\end{enumerate}

An element $\mathbf{\Psi}\in\mathfrak t^*_{\ell,\hgs}$ is called rational if
there exists an $I_0$-tuple of rational functions $(g_i(z))_{i\in I_0}$ such
that each $g_i(z)$ is regular at $z=0$ and $z=\infty$, satisfies
$g_i(0)g_i(\infty)=1$, and such that $\Psi_i^+(z)$ and $\Psi_i^-(z)$ are
respectively the expansions of $g_i(z)$ at $z=0$ and at $z=\infty$. We denote
by $\mathfrak r_{\hgs}$ the set of rational elements of
$\mathfrak t^*_{\ell,\hgs}$.

Following \cite{Hernandeztwisted}, we shall use the following
distinguished rational $\ell$-weights. For
$i\in I_0$ and $a\in\CC^*$, define $Y_{i,a}$ by
\[
(Y_{i,a})_j(z)
=
\begin{cases}
\displaystyle
q^M\frac{1-a^M z^M q^{-M}}{1-a^M z^M q^M},
& \text{if } \sigma(i)=i \text{ and } j=i, \\[1.2ex]
\displaystyle
q\frac{1-azq^{-1}}{1-azq},
& \text{if } \sigma(i)\neq i \text{ and } j=i, \\[1.2ex]
1,
& \text{otherwise.}
\end{cases}
\]
We also define
\[
A_{i,a}
=
\begin{cases}
\begin{aligned}
Y_{i,aq}Y_{i,aq^{-1}}
&\prod_{\substack{j\sim i\\ j=\sigma(j)}}Y_{j,a}^{-1}
 \prod_{\substack{j\sim i\\ j\neq\sigma(j)}}
 \prod_{s=1}^{M}Y_{j,\zeta^s a}^{-1},
\end{aligned}
& \text{if } C_{i,\sigma(i)}=2, \\[2.2ex]
\displaystyle
Y_{i,aq}Y_{i,aq^{-1}}
\prod_{j\sim i}Y_{j,a}^{-1},
& \text{if } C_{i,\sigma(i)}=0, \\[2.2ex]
\displaystyle
Y_{i,aq}Y_{i,aq^{-1}}Y_{i,-a}^{-1}
\prod_{j\sim i}Y_{j,a}^{-1},
& \text{if } C_{i,\sigma(i)}=-1.
\end{cases}
\]
Here all products are over $j\in I_0$, and for $i,j\in I_0$, we write $j\sim i$ if and only if $C_{i,j}=-1$.

The following result is the twisted analogue of
\cite[Theorem~3.7]{Affinization-of-category-mukhin}.

\begin{theorem}
\label{Ohatgsigma-classification}
Let $\mathbf{\Psi}\in\mathfrak t^*_{\ell,\hgs}$. Then all weight spaces of
$L(\mathbf{\Psi})$ are finite-dimensional if and only if $\mathbf{\Psi}$ is
rational. Moreover, the map
\[
\mathbf{\Psi}\longmapsto L(\mathbf{\Psi})
\]
induces a bijection from $\mathfrak r_{\hgs}$ to the set of isomorphism classes
of simple objects of $\mathcal O_{\hgs}$.
\end{theorem}

Let $\mathfrak t^*_{\ell,\bs}$ be the set of $I_0$-tuples
$\mathbf\Psi=(\Psi_i(z))_{i\in I_0}$, where
$\Psi_i(z)\in\CC[[z]]$ and $\Psi_i(0)\neq 0$ for all $i\in I_0$.

We define the category $\mathcal O_{\bs}$ as the full subcategory of
$\mathcal U_q(\bs)$-modules satisfying the same conditions as in the definition
of $\mathcal O_{\hgs}$: namely, a module $V$ belongs to $\mathcal O_{\bs}$ if
it is $\mathfrak t$-diagonalizable, all its weight spaces are finite-dimensional,
and its weights are contained in a finite union of sets of the form
$D(\lambda)=\{\omega\in\mathfrak t^*\mid \omega\leq\lambda\}$.

An element $\mathbf\Psi\in\mathfrak t^*_{\ell,\bs}$ is called rational if
there exists an $I_0$-tuple of rational functions $(g_i(z))_{i\in I_0}$ such
that each $g_i(z)$ is regular and nonzero at $z=0$, and such that $\Psi_i(z)$
is the expansion of $g_i(z)$ at $z=0$. We denote by $\mathfrak r_{\bs}$ the
set of rational elements of $\mathfrak t^*_{\ell,\bs}$.

The following theorem is the Borel counterpart of
Theorem~\ref{Ohatgsigma-classification}.

\begin{theorem}[{\cite[Theorem 4.6]{wang2023QQ}}]
\label{Obsigma-classification}
Let $\mathbf\Psi\in\mathfrak t^*_{\ell,\bs}$. Then all weight spaces of
$L^{\bs}(\mathbf\Psi)$ are finite-dimensional if and only if $\mathbf\Psi$ is
rational. Moreover, the map
\[
\mathbf\Psi\longmapsto L^{\bs}(\mathbf\Psi)
\]
induces a bijection from $\mathfrak r_{\bs}$ to the set of isomorphism classes
of simple objects of $\mathcal O_{\bs}$.
\end{theorem}

From the definition, we have an inclusion of spaces of rational functions $\mathfrak{r}_{\hgs}\hookrightarrow \mathfrak{r}_{\bs}$. The following lemma is analogue to {\cite[Lemma 3.4]{FJMM2017}} and the proof works word by word.

\begin{lemma}
    The restriction functor $\mathrm{Res}:\mathcal{O}_{\hgs}\to \mathcal{O}_{\bs}$ sends simple objects to simple objects.
\end{lemma}

As the same argument in \cite[Remark 4.4]{Hernandezshfted}, for  $i\in I_0$, the action of $\phi^+_i(z)$ and $\phi^-_i(z)$ coincide on a representation $V\in \mathcal{O}_{\hgs}$, seen as rational operators on each weight space. All $\ell$-weights of $V\in \mathcal{O}_{\hgs}$ are rational (See also in \cite[Section~3.6]{GTL}).

\begin{proposition} [\cite{Affinization-of-category-mukhin,mukhintwisted}]
Let $V\in \mathcal O_{\hat{\mathfrak g}^{\sigma}}$, and fix $i\in I_0$.
Let $\eta,\nu$ be two $\ell$-weights of $V$ such that
\[
x_{i,r}^{\pm}(V_\eta)\cap V_\nu\neq \{0\}
\qquad\text{for some }r\in \mathbb Z.
\]
Then the following hold.
 
(i) There exists $a\in \mathbb C^*$ such that $\nu=\eta A_{i,a}^{\pm1}.$

(ii) There exist bases $(v_k)_{1\le k\le \dim V_\eta}$ of $V_\eta$ and
$(w_\ell)_{1\le \ell\le \dim V_\nu}$ of $V_\nu$ such that, for each $k$,
the $\nu$-component of $x_i^\pm(z)v_k$ is of the form
\[
(x_i^\pm(z)v_k)_\nu=
\sum_{\ell=1}^{\dim V_\nu} w_\ell\,
P_{k,\ell}^\pm\!\left(\frac{\partial}{\partial a^{N_i}}\right)\delta\!\left(\frac{a^{N_i}}{z^{N_i}}\right),
\]
with $\deg P_{k,\ell}^\pm\le k+\ell-2$.

\end{proposition}

\begin{proof}
The proof follows the arguments of \cite[Proposition~3.8]{Affinization-of-category-mukhin} and \cite[Proposition~3.6]{mukhintwisted}. We give the details, since some modifications are needed in the present setting.

Firstly, we consider $M=2$ and we treat the case of $x_i^+(z)$; the argument for $x_i^-(z)$ is identical.

Choose a basis $(v_k)_{1\le k\le \dim V_\eta}$ of $V_\eta$ in which all
operators $\phi_{j,r}^+$ are simultaneously upper-triangular, and a basis $(w_\ell)_{1\le \ell\le \dim V_\nu}$ of $V_\nu$ in which they are
simultaneously lower-triangular. Thus for every $j\in I_0$ one has
\[
(\phi _{j}^{+}(u)-\eta _{j}^{+}(u)).v_k=\sum_{k^{\prime}<k}{v_{k^{\prime}}\xi _{j}^{+,k,k^{\prime}}(u)},
\qquad
(\phi _{j}^{+}(u)-\nu _{j}^{+}(u)).w_l=\sum_{l^{\prime}>l}{w_{l^{\prime}}\zeta _{j}^{+,l,l^{\prime}}(u)},
\]
with $\xi_j^{+,k,k'}(u),\zeta_j^{+,\ell,\ell'}(u)\in u\mathbb C[[u]]$.

For each $1\le k\le \dim V_\eta$, we write
\[
(x_i^+(z)v_k)_\nu=\sum_{\ell=1}^{\dim V_\nu}\lambda_{k,\ell}(z)w_\ell,
\qquad
\lambda_{k,\ell}(z)\in \mathbb C[[z,z^{-1}]].
\]

By the relation, we have 
\begin{align*}
&(q^{C_{j,\sigma(i)}}+uz)(q^{C_{ji}}-uz)\,x_i^{+}(z)\,
\bigl(\phi_j^{+}(u)-\eta_j^{+}(u)\bigr).v_k \notag\\
& 
= \phi_j^{+}(u)\,x_i^{+}(z)\,
(1+q^{C_{j,\sigma(i)}}uz)(1-q^{C_{ji}}uz)\,.v_k -(q^{C_{j,\sigma(i)}}+uz)(q^{C_{ji}}-uz)\,x_i^{+}(z)\,\eta_j^{+}(u)\,.v_k .
\end{align*}
By taking the $w_l$ component, we have
\begin{align}
    &(q^{C_{j,\sigma (i)}}+uz)(q^{C_{ji}}-uz)\sum_{k^{\prime}=1}^{k-1}{\xi _{j}^{+,k,k^{\prime}}(u)}\lambda _{k^{\prime},l}(z)\nonumber\\
    &=((1+q^{C_{j,\sigma (i)}}uz)(1-q^{C_{ji}}uz)\nu _{j}^{+}(u)-(q^{C_{j,\sigma (i)}}+uz)(q^{C_{ji}}-uz)\eta _{j}^{+}\left( u \right))\lambda _{k,l}(z)\nonumber\\
    &+(1+q^{C_{j,\sigma (i)}}uz)(1-q^{C_{ji}}uz)\sum_{l^\prime =1}^{l-1}{\lambda _{k,l^{\prime}}(z)}\zeta _{j}^{+,l^{\prime},l}(u). \label{FJMMlemma5.5} 
\end{align}
Since $( x_{i}^{+}( z ) .V_{\eta}) _{\nu}\ne 0$, we may choose the smallest $K$ such that $( x_{i}^{+}( z ) .v_K) _{\nu}\ne 0$ and then a smallest $L$ such that $\lambda_{K,L}(z)\ne 0$. So (\ref{FJMMlemma5.5}) gives
\begin{align}
    0=((1+q^{C_{j,\sigma (i)}}uz)(1-q^{C_{ji}}uz)\nu _{j}^{+}(u)-(q^{C_{j,\sigma (i)}}+uz)(q^{C_{ji}}-uz)\eta _{j}^{+}\left( u \right))\lambda _{K,L}(z).\label{FJMMlemma5.5-2}
\end{align}
This  must hold for all $j\in I_0$. As the same argument in \cite[Proposition~3.6]{mukhintwisted}, for all $j\in I_0$, there exists $a\in \CC^*$ such that 
\begin{align*}
    \nu _{j}^{+}(u)\left( \eta _{j}^{+}\left( u \right) \right) ^{-1}=\frac{(q^{C_{j,\sigma (i)}}+ua)(q^{C_{ji}}-ua)}{(1+q^{C_{j,\sigma (i)}}ua)(1-q^{C_{ji}}ua)}=\left( A_{i,a}\left( u \right) \right)_j
\end{align*}
as an equality of power series in $u$. This establishes $(i)$.

We now  turn to $(ii)$. Choose $j=i$ and let $\alpha :=q^{C_{i,\sigma (i)}}$, $\beta :=q^{C_{i,i}}$.
If $i\ne \sigma(i)$,
by (\ref{FJMMlemma5.5}), we have 
\begin{align}
    &\frac{(1+\alpha \beta )\eta _{i}^{+}\left( u \right) u}{(1+\alpha ua)(1-\beta ua)}\Bigl( (\alpha -\beta )(1+u^2az)+(1-\alpha \beta )u(z+a) \Bigr) (z-a)\lambda _{k,l}(z)\nonumber\\
    &=(\alpha +uz)(\beta -uz)\sum_{k^{\prime}=1}^{k-1}{\xi _{i}^{+,k,k^{\prime}}(u)\lambda _{k^{\prime},l}(z)}-(1+\alpha uz)(1-\beta uz)\sum_{l^\prime =1}^{l-1}{\lambda _{k,l^{\prime}}(z)\zeta _{i}^{+,l^{\prime},l}(u).}\label{FJMMlemma5.5-3}
\end{align}

We prove the desired form of $\lambda_{k,l}(z)$ by induction on $k+l$. If $k+l=2$, then the right hand of (\ref{FJMMlemma5.5-3}) is zero. Taking the coefficient of $u^1$, we obtain $(z-a)\lambda_{1,1}(z)=0$, Hence  $\lambda_{1,1}(z)=P_{1,1}\delta(a/z)$ with $P_{1,1}\in \CC$. Assume now that the claim is proved whenever $k'+l^\prime<k+\ell$.
Taking the coefficient of $u^1$ in (\ref{FJMMlemma5.5-3}), we have 
\begin{align*}
    \eta^+_{i,0}(1+\alpha\beta)(\alpha-\beta)(z-a)\lambda_{k,l}(z)=\alpha\beta\sum_{k^{\prime}=1}^{k-1}{\xi _{i,1}^{+,k,k^{\prime}}\lambda _{k^{\prime},l}(z)}-\sum_{l^\prime =1}^{l-1}{\lambda _{k,l^{\prime}}(z)\zeta _{i,1}^{+,l^{\prime},l}}.
\end{align*}
Notice that for all $m\ge1$, we have $\left( z-a \right) \frac{1}{m!}\left( \frac{\partial}{\partial a} \right) ^m\delta (z/a)=\frac{1}{\left( m-1 \right) !}\left( \frac{\partial}{\partial a} \right) ^{m-1}\delta (z/a)$. Then every solution  is of the form $\lambda _{k,l}\left( z \right) =P_{k,l}\left( \frac{\partial}{\partial a} \right) \delta \left( \frac{a}{z} \right)$ with $\mathrm{deg(}P_{k,\ell})\le 1+\max_{\{
k^{\prime}<k,\ell ^{\prime}<\ell\\
\}} \left( \mathrm{deg(}P_{k^{\prime},\ell ^{\prime}}) \right)$.

If $i=\sigma(i)$, by (\ref{FJMMlemma5.5}), we have 
\begin{align}
    &\frac{(1 - \alpha^4)\eta _{i}^{+}\left( u \right) u^2}{1 - \alpha^2 u^2 a^2}(z^2 - a^2)\lambda _{k,l}(z)\nonumber\\&=(\alpha^2-u^2z^2)\sum_{k^{\prime}\nonumber =1}^{k-1}{\xi _{i}^{+,k,k^{\prime}}(u)\lambda _{k^{\prime},l}(z)}-(1-\alpha^2u^2z^2)\sum_{l^\prime =1}^{l-1}{\lambda _{k,l^{\prime}}(z)\zeta _{i}^{+,l^{\prime},l}(u)}\label{FJMMlemma5.5-4}.
\end{align}
Notice for $i=\sigma(i)$, $\lambda _{k,l}(z)\in \CC[[z^2,z^{-2}]]$, $\xi _{i}^{+,k,k^{\prime}}(u)\in u^2\mathbb{C} [[u^2]]$, $\zeta _{i}^{+,l,l^{\prime}}(u)\in u^2\CC[[u^2]]$ for all $1\le k\le \dim V_{\eta}$ and $1\le l\le \dim V_\nu$. Use the similar argument, we have $\lambda _{k,l}\left( z \right) =P_{k,l}\left( \frac{\partial}{\partial a^2} \right) \delta \left( \frac{a^2}{z^2} \right)$ with $\mathrm{deg(}P_{k,\ell})\le 1+\max_{\{
k^{\prime}<k,\ell ^{\prime}<\ell\\
\}}\left( \mathrm{deg(}P_{k^{\prime},\ell ^{\prime}}) \right)$.

For $M=3$, the argument is similar. Thus we complete the proof.   
\end{proof}

The preceding proposition implies the following description of the action of
the currents $x_i^\pm(z)$ on objects of $\mathcal O_{\hgs}$.
\begin{lemma}
\label{xaction}
Let $V\in\mathcal O_{\hgs}$, let $\eta$ be a $\ell$-weight occurring in $V$, and let $v\in V_\eta$. Then the formal series
$x_i^\pm(z)v$ can be written in the form
\[
x_i^\pm(z)v
=
\sum_a\sum_{k=0}^{r_a}
v_{k,a}
\left(\frac{\partial}{\partial a^{N_i}}\right)^k
\delta\left(\frac{a^{N_i}}{z^{N_i}}\right),
\]
where $a$ runs over a finite subset of $\CC^*$,
\[
v_{k,a}\in V_{\eta A_{i,a}^{\pm1}},
\qquad
r_a=\dim V_\eta+\dim V_{\eta A_{i,a}^{\pm1}}-2 .
\]
\end{lemma}

\begin{theorem}[{\cite[Theorem~5.4]{wang2023QQ}}]
\label{sufficientlylarge}
Let $\mathbf\Psi=(\psi_i(u))_{i\in I_0}$ be a polynomial $\ell$-weight, that is,
each $\psi_i(u)$ is a polynomial in $u$ with nonzero constant term. Then, for
all $v\in L^{\bs}(\mathbf\Psi)$, $i\in I_0$ and $\alpha\in\Delta^+$, one has
\[
E_{p\delta\pm\alpha}v=0,
\qquad
\phi_{i,p}^+v=0
\]
for all sufficiently large $p$.
\end{theorem}

\begin{corollary}
\label{1-finite}
Every simple module with polynomial highest $\ell$-weight is isomorphic to a
tensor product of finitely many positive prefundamental modules and a
one-dimensional module. In particular, it is $1$-finite; that is, the family
of currents
\[
\{k_i^{-1}\phi_i^+(z)\}_{i\in I_0}
\]
has a single joint eigenvalue on the module.
\end{corollary}

\begin{proof}
The proof is parallel to the untwisted case
\cite[Theorem~4.11]{FrenkelHernandez-Bax}. The $1$-finiteness follows from the proof of \cite[Theorem~4.24]{wang2023QQ}.
\end{proof}

\begin{lemma}
\label{Drinfeld-coproduct-module}
Let $V\in\mathcal O_{\hgs}$ and $M=L^{\bs}(\boldsymbol{\Psi})$ with $\boldsymbol{\Psi}$ polynomial $\ell$-weight. Then the Drinfeld
coproduct endows $V\otimes M$ with a well-defined
$\mathcal U_q(\bs)$-module structure. We denote the resulting module by
$V\otimes_D M$.
\end{lemma}

\begin{proof}
Together with Theorem~\ref{sufficientlylarge}, the argument is the same as in \cite[Lemma~5.6]{FJMM2017}.
\end{proof}

The next lemma shows the submodules  of $V\otimes_D M$ has a very special form.

\begin{lemma}\label{submodule-structure}
    Let $W\subset V\otimes_D M$ be a nonzero submodule. Then there exists a linear subspace $V^{(0)}\subset V$ such that $W=V^{(0)}\otimes M$. The highest $\ell$-weight vector $v_0$ of $V$ belongs to $V^{(0)}$.
\end{lemma}

\begin{proof}
The proof is adapted from \cite[Lemma~5.7]{FJMM2017}. By the same argument in \cite[Lemma~5.7]{FJMM2017}, we have 

\begin{equation}
\label{x^pmcontain}
\begin{aligned}
    &\left( \sum_{j\ge 0}{x_{i,k-j}^{-}}\otimes \phi _{i,j}^{+} \right) W\subset W,
    \quad  (1\otimes x)W\subset W
    \quad(x\in \mathcal{U}_q(\mathfrak{b}^\sigma)^-),\\
    &(x_{i,k}^{+}\otimes 1)W\subset W,
    \quad
    \left( \sum_{j\ge 0}{\phi _{i,-j}^{-}}\otimes x_{i,k+j}^{+} \right) W\subset W.
\end{aligned}
\end{equation}
Here $i\in I$, $k$ is in the range for which the corresponding Drinfeld
generator belongs to the Borel subalgebra.

Let $m_0$ be the highest $\ell$-weight vector of $M$, and set
\[
V^{(0)}
=
\{v\in V\mid v\otimes m_0\in W\}.
\]
This is a linear subspace of $V$. Since $M$ is generated by $m_0$ under
$\mathcal{U}_q(\mathfrak b^\sigma)^-$, the inclusion
$(1\otimes x)W\subset W$ in \eqref{x^pmcontain} gives
\[
V^{(0)}\otimes M\subset W .
\]
It remains to prove the reverse inclusion. We shall prove the following assertion. Suppose that
\[
w=\sum_{r=1}^{N}v_r\otimes m_r\in W,
\]
where $v_r\in V$ and $m_1,\ldots,m_N$ are linearly independent weight
vectors of $M$. Then $v_r\in V^{(0)}$ for all $1\le r\le N$.

We prove this assertion by induction on $N$. If $N=1$, write $w=v_1\otimes m_1$. If $m_1$ is proportional to $m_0$, then
$v_1\in V^{(0)}$ by definition.  Otherwise, there exists an $i\in I$ and $k\ge0$ such that $x^+_{i,k}.m_1\ne 0$. By Theorem \ref{sufficientlylarge}, there exists the largest $k$ such that $x^+_{i,k}.m_1\ne 0$. Applying the second
operator in the second line of \eqref{x^pmcontain}, all terms with $j>0$
vanish, and hence
\[
\phi_{i,0}^{-}v_1\otimes x_{i,k}^{+}m_1\in W .
\]
Repeating the same argument finitely many
times, we obtain $v_1\otimes m_0\in W$. Thus $v_1\in V^{(0)}$.
Assume that the claim holds for $N^\prime<N$. 
We can find an $i\in I$ and $k\ge 0$ such that $x^+_{i,k}m_r\ne 0$ for some $r$ and $x^+_{i,l}m_s=0$ for all $l>k$, $1\le s\le N$. By (\ref{x^pmcontain}), $\sum_{r=1}^N{\phi_{i,0}^{-}v_r\otimes x_{i,k}^{+}m_r}\in W$. If the vectors $x_{i,k}^{+}m_1,\ldots,x_{i,k}^{+}m_N$ are linearly independent, then we repeat this procedure. After a finite number of steps, we have $\sum_{r=1}^N{v_r\otimes m_{r}'}\in W$, $m_r'\ne0$ for some $r$ and $m_1',\ldots,m_N'$ are not linearly independent. After reordering the indices, we can assume that $m_1',\ldots,m_{N'}'$ are linearly independent and $m_{r}^{\prime}=\sum_{s=1}^{N^{\prime}}{a_{r,s}m_{s}^{\prime}}$, $N^{\prime}<r<N$. Then 
$\sum_{r=1}^N v_r\otimes m_r^\prime=
\sum_{s=1}^{N^\prime}\left(v_s+\sum_{r=N^\prime+1}^N a_{r,s}v_r \right)\otimes m_s^\prime$.

By the induction hypothesis, we have $v_s+\sum_{r=N^\prime+1}^N a_{r,s}v_r \in V^{(0)}$, $1\le s\le N^\prime$. Thus $\sum_{s=1}^{N'}
\left(v_s+\sum_{r=N'+1}^{N}a_{r,s}v_r\right)\otimes m_s
	\in V^{(0)}\otimes M\subset W $.
Subtracting this element from the original vector $w$, we obtain
\[
\sum_{r=N'+1}^{N}
v_r\otimes
\left(m_r-\sum_{s=1}^{N'}a_{r,s}m_s\right)
\in W .
\]
By the induction hypothesis again, we have $v_r\in V^{(0)}(N'<r\le N )$. Since $v_s=v_{s}^{\prime}-\sum_{r=N^{\prime}+1}^N{a_{r,s}v_r}(1\le s\le N^\prime)$, we have $v_s\in V^{(0)}(1\le r\le N^\prime)$. 

Then we have $W=V^{(0)}\otimes M$. Since $W\ne 0$, we can choose a nonzero $v\in V^{(0)}$. Moreover, by (\ref{x^pmcontain}), we have $(x_{i,k}^{+}\otimes 1)(v\otimes m_0)=x_{i,k}^{+}v\otimes m_0\in W$. This implies $x^+_{i,k}V^{(0)}\subset V^{(0)}$ for all $i\in I, k\ge 0$. This shows $v_0\in V^{(0)}$.    
\end{proof}

\section{Shifted twisted quantum affine algebras}
In this section, we introduce shifted twisted quantum affine algebras.
Starting from the Drinfeld presentation of the twisted quantum affine algebra
$\mathcal U_q(\lgs)$, we define the shifted twisted quantum algebra
$\mathcal U_q^{\mu_+,\mu_-}(\hat{\mathfrak g}^{\sigma})$ for 
$\mu_+,\mu_-\in\Lambda$ by modifying the leading terms of the Cartan
currents. We establish the triangular
decomposition (Proposition \ref{triangular}), and show that the algebra depends, up to isomorphism, only
on the sum $\mu=\mu_++\mu_-$ (Theorem \ref{up-to-iso}). Finally, in type $A_2^{(2)}$, we obtain an
embedding of the Borel subalgebra into the shifted twisted qauntum affine algebra $\U_q^\mu(A_2^{(2)})$ for $\mu$ antidominant (Proposition \ref{borel-injective}).

\subsection{Definition of shifted twisted quantum affine algebras}
	
For $\hgs=A_{2n}^{(2)}\,(n\ge 2)$,
$A_{2n-1}^{(2)}\,(n\ge 2)$, $D_{n+1}^{(2)}\,(n\ge 2)$,
$E_6^{(2)}$, and $D_4^{(3)}$, the Drinfeld presentation of the twisted
quantum affine algebra $\mathcal U_q(\lgs)$ was
recalled in Section~$1$. Let $\Lambda'=\bigoplus_{i\in I_0}\mathbb Z\omega_i^\vee$ be the coweight lattice of $\mathfrak g^\sigma$.
	
\begin{definition}\label{basic-definition}
Let  $M$ be the order of $\sigma$, $\mu_{+}$, $\mu_{-}\in \Lambda^\prime$ such that $M\mid \alpha_i\left( \mu_{\pm}\right)$ if $i=\sigma(i)$. Fix $\zeta=\exp\left(\frac{\mathrm{i}2\pi}{M} \right)$, ${\rm i} = \sqrt{-1}$. The shifted twisted quantum affine algebra
$\mathcal U_q^{\mu_+,\mu_-}(\hat{\mathfrak g}^{\sigma})$
is the associative $\mathbb C$-algebra generated by $x_{i,m}^{\pm}$, $\phi_{i,m}^{\pm}\left(i\in I_0, m\in \mathbb{Z} \right) $, $h_{i,r}\left(i\in I_0, r\in \mathbb{Z}\backslash \{0\} \right) $, subject to the following relations:

	\begin{align}
		&x_{i,r}^{\pm}=0,
		\qquad (i,r)\in (I_0\times\mathbb Z)\setminus I_{\mathbb Z},\\
		&h_{i,m}=0,
		\qquad (i,m)\in (I_0\times(\mathbb Z\setminus\{0\}))\setminus I_{\mathbb Z},\\
		&[\phi_{i,r}^{\pm},\phi_{j,r'}^{\pm}]
		=
		[\phi_{i,r}^{+},\phi_{j,r'}^{-}]
		=
		[\phi_{i,r}^{\pm},h_{j,m}]
		=
		[h_{i,m},h_{j,m'}]
		=0,\\
		&\phi_{i,-\alpha_i(\mu_+)}^+x_{j,r}^{\pm}
		=
		q^{\pm\sum_{s=1}^{M}C_{i,\sigma^s(j)}}
		x_{j,r}^{\pm}\phi_{i,-\alpha_i(\mu_+)}^+,\\
		&\phi_{i,\alpha_i(\mu_-)}^-x_{j,r}^{\pm}
		=
		q^{\mp\sum_{s=1}^{M}C_{i,\sigma^s(j)}}
		x_{j,r}^{\pm}\phi_{i,\alpha_i(\mu_-)}^-,\\
		&[h_{i,m},x_{j,r}^{\pm}]
		=
		\pm \frac{1}{m}
		\left(
		\sum_{s=1}^{M}
		\left[
		\frac{mC_{i,\sigma^s(j)}}{d_{i}}
		\right]_{q_{i}}
		\zeta^{ms}
		\right)
		x_{j,m+r}^{\pm},\label{hx-relation}\\
		&[x_{i,r}^{+},x_{j,r'}^{-}]
		=
		\frac{\sum_{s=1}^{M}\delta_{\sigma^s(i),j}\zeta^{sr'}}{N_i}
		\frac{\phi_{i,r+r'}^{+}-\phi_{i,r+r'}^{-}}
		{q_{i}-q_{i}^{-1}}.\label{xphi-relation}
	\end{align}
	Here the elements $\phi_{i,\pm r}^{\pm}$ are defined by the generating series
	\[
	\phi_i^{\pm}(z)
	=
	\sum_{r\in\mathbb Z}\phi_{i,\pm r}^{\pm}z^{\pm r}
	=
	z^{\mp\alpha_i(\mu_\pm)}
	\phi_{i,\mp\alpha_i(\mu_\pm)}^{\pm}
	\exp\left(
	\pm(q_{i}-q_{i}^{-1})
	\sum_{m\geq 1}h_{i,\pm m}z^{\pm m}
	\right),
	\]
with $\phi^\pm_{i,\mp\alpha_i(\mu_{\pm})}$ ($i\in I_0$) which are invertible.

Set $x_i^{\pm}(u)=\sum_{r\in\ZZ}x_{i,r}^{\pm}u^{-r}$. In terms of these currents, the remaining Drinfeld--Serre relations are as follows. For all $i,j\in I_0$,
	\begin{align}
		&\prod_{s=1}^{M}
		\left(u_1-\zeta^s q^{\pm C_{i,\sigma^s(j)}}u_2\right)
		x_i^{\pm}(u_1)x_j^{\pm}(u_2)=
		\prod_{s=1}^{M}
		\left(u_1q^{\pm C_{i,\sigma^s(j)}}-\zeta^s u_2\right)
		x_j^{\pm}(u_2)x_i^{\pm}(u_1).\label{xx-relation}
	\end{align}
	If $C_{i,j}=-1$ and $\sigma(i)\neq j$, then
	\begin{align}
	\operatorname{Sym}_{u_1,u_2}
		\Bigl\{
		P_{i,j}^{\pm}(u_1,u_2)
		\bigl(
		&x_j^{\pm}(v)x_i^{\pm}(u_1)x_i^{\pm}(u_2)\label{Serre1}\\
		&-(q^{2Md_{i,j}}+q^{-2Md_{i,j}})
		x_i^{\pm}(u_1)x_j^{\pm}(v)x_i^{\pm}(u_2)\notag\\
		&+x_i^{\pm}(u_1)x_i^{\pm}(u_2)x_j^{\pm}(v)
		\bigr)
		\Bigr\}=0 \notag.
	\end{align}
	If $C_{i,\sigma(i)}=-1$, then
	\begin{align}
		&\operatorname{Sym}_{u_1,u_2,u_3}
		\Bigl\{
		\bigl(
		q^{3/2}u_1^{\mp1}
		-(q^{1/2}+q^{-1/2})u_2^{\mp1}
		+q^{-3/2}u_3^{\mp1}
		\bigr)
		x_i^{\pm}(u_1)x_i^{\pm}(u_2)x_i^{\pm}(u_3)
		\Bigr\}=0,\label{Serre2}\\
		&\operatorname{Sym}_{u_1,u_2,u_3}
		\Bigl\{
		\bigl(
		q^{-3/2}u_1^{\pm1}
		-(q^{1/2}+q^{-1/2})u_2^{\pm1}
		+q^{3/2}u_3^{\pm1}
		\bigr)
		x_i^{\pm}(u_1)x_i^{\pm}(u_2)x_i^{\pm}(u_3)
		\Bigr\}=0 \label{Serre3}.
	\end{align}
\end{definition}

It follows from \cite[Remark~2.7]{wang2023QQ} and \cite{chen2023twisted}
that the Drinfeld currents satisfy the following relations:

\begin{align}
    &\phi_i^+(z)x_j^\pm(w)=x_j^\pm(w)\phi_i^+(z)g_{ij}(zw)^{\pm1},\label{phi+x}\\
    &\phi_i^-(z)x_j^\pm(w)=x_j^\pm(w)\phi_i^-(z)g_{ji}(1/zw)^{\mp1},\label{phi-x}\\
    &\left[ x_{i}^{+}\left( z \right) ,x_{j}^{-}\left( w \right) \right] =\frac{1}{N_i(q_{i}-q_{i}^{-1})}\sum_{s=1}^M{\delta _{\sigma ^s\left( i \right) ,j}}\label{x+x-}\\
    &\qquad \quad\times\left( \phi _{i}^{+}\left( z^{-1} \right) \delta \left( \frac{\zeta ^{-s}w}{z} \right) -\phi _{i}^{-}\left( z^{-1} \right) \delta \left( \frac{\zeta ^{-s}w}{z} \right) \right),\notag
\end{align}
where 
\begin{align*}
g_{ij}(z)=\frac{\prod_{s=1}^M{\left( q^{C_{i,\sigma ^s\left( j \right)}}-\zeta ^sz \right)}}{\prod_{s=1}^M{\left( 1-\zeta^sq^{C_{i,\sigma ^s\left( j \right)}}z \right)}}, \quad \delta \left( z \right) =\sum_{n\in \mathbb{Z}}{z^n}.
\end{align*}
A direct computation gives $g_{ij}(z)=g_{ji}(z^{-1})^{-1}$.

\begin{remark}
For later use, we set
\[
\Lambda=
\left\{\mu\in\Lambda'\ \middle|\ M\mid \alpha_i(\mu)
\text{ whenever } \sigma(i)=i
\right\}.
\]
We also define
\[
\Lambda^+
=
\left\{
\mu\in\Lambda
\ \middle|\
\alpha_i(\mu)\ge 0
\text{ for all } i\in I_0
\right\}.
\]
An element $\mu\in\Lambda$ is called dominant if $\mu\in\Lambda^+$. Thus the $\mu_+$ and $\mu_-$ appearing below will always be assumed
to belong to $\Lambda$.
\end{remark}

The following triangular decomposition is proved by the same argument as
\cite[Theorem~4.2]{chen2023twisted}.

\begin{proposition}\label{triangular}
	Let
	$\mathcal U_q^{\mu_+,\mu_-}(\hat{\mathfrak g}^{\sigma})^{+}$
	be the subalgebra of
	$\mathcal U_q^{\mu_+,\mu_-}(\hat{\mathfrak g}^{\sigma})$
	generated by the elements $x_{i,m}^{+}$, and let
	$\mathcal U_q^{\mu_+,\mu_-}(\hat{\mathfrak g}^{\sigma})^{-}$
	be the subalgebra generated by the elements $x_{i,m}^{-}$. Let
	$\mathcal U_q^{\mu_+,\mu_-}(\hat{\mathfrak h}^{\sigma})$
	be the Cartan subalgebra generated by the elements $\phi_{i,m}^{\pm}$.
	Then the multiplication map
	\[
	\mathcal U_q^{\mu_+,\mu_-}(\hat{\mathfrak g}^{\sigma})^{-}
	\otimes
	\mathcal U_q^{\mu_+,\mu_-}(\hat{\mathfrak h}^{\sigma})
	\otimes
	\mathcal U_q^{\mu_+,\mu_-}(\hat{\mathfrak g}^{\sigma})^{+}
	\longrightarrow
	\mathcal U_q^{\mu_+,\mu_-}(\hat{\mathfrak g}^{\sigma})
	\]
	is an isomorphism of vector spaces.
\end{proposition}
	
We obtain the following theorem, as with the homomorphism defined in \cite[Section $(5)$(i)]{FT}.
	
\begin{theorem}\label{up-to-iso}
	The algebra
	$\mathcal U_q^{\mu_+,\mu_-}(\hat{\mathfrak g}^{\sigma})$
	depends, up to isomorphism, only on the sum
	\[
	\mu=\mu_+ + \mu_- .
	\]
	More precisely, there is an algebra isomorphism
	\[
	\mathcal U_q^{\mu_+,\mu_-}(\hat{\mathfrak g}^{\sigma})
	\cong
	\mathcal U_q^{0,\mu_++\mu_-}(\hat{\mathfrak g}^{\sigma}) .
	\]
\end{theorem}

\begin{remark}\label{central-remark}
$(i)$ For $i\in I_0$, the product $\phi _{i,-\alpha _i\left( \mu + \right)}^{+}\phi _{i,\alpha _i\left( \mu - \right)}^{-}$ is central. The twisted quantum loop algebra $\mathcal{U}_q(\mathcal{L}\gs)$ is the quotient of $\mathcal{U}_q^{0,0}(\hgs)$ by idetifying $\phi _{i,0}^{+}\phi _{i,0}^{-}$ with $1$ for $i\in I$.

$(ii)$ In view of Theorem~\ref{up-to-iso}, we will simply denote $\mathcal{U} _{q}^{\mu}\left( \hat{\mathfrak{g}}^{\sigma} \right)=\mathcal{U} _{q}^{0,\mu}\left( \hat{\mathfrak{g}} ^{\sigma} \right).$
\end{remark}

Following the construction of the asymptotic algebra
$\widetilde{\mathcal U}_q(A_2^{(2)})$ in \cite{Garbali-Alexander},
and arguing as in \cite[Section~3.2]{Hernandezshfted}, we obtain the
following embedding result.

\begin{proposition}\label{borel-injective}
	Let $\mu\in -\Lambda^+$. Then there exists an injective algebra homomorphism
	\[
	I_\mu^\sigma:\mathcal U_q(\mathfrak b^\sigma)
	\hookrightarrow
	\mathcal U_q^\mu(A_2^{(2)}).
	\]
	Equivalently, $\mathcal U_q^\mu(A_2^{(2)})$ contains a subalgebra isomorphic
	to $\mathcal U_q(\mathfrak b^\sigma)$.
\end{proposition}
	
		

\section{\texorpdfstring{Category $\mathcal{O}^{sh}$}{Category O\^{}sh}}
	
In this section, we introduce and study the categories
$\mathcal O_\mu$ of representations of shifted twisted quantum affine
algebras. We define their weight and $\ell$-weight decompositions, classify
their simple objects by rational $\ell$-weights, and assemble them into the
direct sum category $\mathcal O^{sh}$ (Theorem \ref{rational-theorem}). We also study shift functors induced by shift homomorphisms.

\subsection{\texorpdfstring{The category $\mathcal{O}_{\mu}$ for the shifted twisted quantum affine algebras}{The category Omu for the shifted twisted quantum affine algebras}}
	
For $\mu_+,\mu_-\in \Lambda$, we introduce the following category.
\begin{definition}
The category $\mathcal O_{\mu_+,\mu_-}$ is the full subcategory of
$\mathcal U_q^{\mu_+,\mu_-}(\hat{\mathfrak g}^{\sigma})$-modules consisting
of modules $V$ satisfying the following conditions.

(i) $V$ admits weight decompositions
\[V=\bigoplus_{\omega\in\mathfrak t^*}V_\omega^+=\bigoplus_{\omega\in\mathfrak t^*}V_\omega^-,
\]
where 
\[
V_\omega^+=\left\{
v\in V\ \middle|\ \phi_{i,-\alpha_i(\mu_+)}^+v=\omega(i)v \text{ for all } i\in I_0\right\}
\]
and 
\[
V_\omega^-=\left\{v\in V\ \middle|\ \phi_{i,\alpha_i(\mu_-)}^-v=\omega(i)^{-1}v\text{ for all } i\in I_0\right\}.
\]

(ii) All weight spaces $V_\omega^+$ and $V_\omega^-$ are finite-dimensional.

(iii) There exist finitely many elements
$\omega_1,\ldots,\omega_s\in\mathfrak t^*$ such that, whenever $V_\omega^+\neq 0$ or $V_\omega^-\neq 0$, one has $\omega\in D(\omega_1)\cup\cdots\cup D(\omega_s),$
where $D(\omega_i)=\{\omega\in\mathfrak t^*\mid \omega\leq \omega_i\}.$
\end{definition}
	
\begin{remark}
$(i)$ In general $V_{\omega}^{+}$ and $V_{\omega}^{-}$ do not coincide  as $\phi^{+}_{i, -\alpha_i(\mu_+)}$ and $\phi^{-}_{i, \alpha_i(\mu_-)}$ are not inverse to each other.
		
$(ii)$ In view of Theorem~\ref{up-to-iso}, the category
$\mathcal O_{\mu_+,\mu_-}$ depends, up to equivalence, only on the sum
$\mu=\mu_++\mu_- .$
Therefore, we shall simply write
$ \mathcal O_\mu:=\mathcal O_{0,\mu}.$
\end{remark}
	
\begin{proposition}
From the definition of $\overline{\alpha_j}$ and defining relations, we have $\phi^{\pm}_{j,m}(V_{\omega}^{\pm})\subseteq V_\omega^\pm$, $x_{j,m}^{+}\left( V_{\omega}^{\pm} \right) \subseteq V_{\omega \overline{\alpha _j}}^{\pm}$,  $x_{j,m}^{-}\left( V_{\omega}^{\pm} \right) \subseteq V_{\omega \overline{\alpha _j}^{-1}}^{\pm}$. 	
\end{proposition}
\begin{proof}
For any $v\in V_\omega^{+}$,  $\phi_{j,m}^{\pm}(j\in I_0, m\in \ZZ)$ and $x_{j,m}^{+}(j\in I_0,m\in\ZZ)$, we have 
\begin{align*}
\phi_{i,-\alpha_i(\mu_+)}^{+}(\phi_{j,m}^{\pm}. v)&=\phi_{j,m}^\pm(\phi_{i,-\alpha_i(\mu_+)}^{+}. v)=\omega(i)\phi_{j,m}^{\pm}. v,\\
\phi_{i,-\alpha_i(\mu_+)}^{+}(x_{j,m}^{+}\cdot v)&=(\overline{\alpha _j}\left( i \right)) \omega(i)x_{j,m}^\pm. v.
\end{align*}  
Thus $\phi_{j,m}^{\pm}(V_{\omega}^{+})\subseteq V_{\omega}^{+}$  and  $x_{j,m}^+ (V_{\omega}^{+})\subseteq V_{\omega \overline{\alpha _j}}^{+}$.		
Similarly, we have the proposition of $x^-_{j,m}$.
\end{proof}	
Let $\mathfrak{r}$ be the group (under pointwise multiplication of rational functions) 
\[
\mathfrak{r} =\{\mathbf\Psi =(\Psi _i(z))_{i\in I_0}\in (\mathbb{C} (z))^{I_0}\mid \Psi _i(z)~\text{is regular at }z=0 ~\text{and}~  \Psi _i(0)\ne 0~ \text{for all}~i\in I_0\}
\]
and define $\mathfrak{r}_\mu=\{\mathbf \Psi=(\Psi_i(z))\in \mathfrak{r}\mid \deg \Psi_i(z)=\alpha_i(\mu) ~\text{for all}~i\in I_0\}$. 
	
For a $\mathcal{U} _{q}^{\mu _+,\mu _-}\left( \hat{\mathfrak{g}} ^{\sigma} \right)$-module $V$,  the linear subspace 
\begin{align*}
V_{\boldsymbol\Psi}
=
\left\{
v\in V
\ \middle|\
\text{for all } i\in I_0 \text{ and } m\in \NN,
\text{ there exists } p\ge1
\text{ such that }
(\phi_{i,m}^+-\Psi_{i,m}^+)^p v=0
\right\}
\end{align*}
is called the $\ell$-weight space of $V$ of $\ell$-weight $\mathbf{\Psi}$, where we use the expansion $\Psi _i\left( z \right) =\sum_{r\ge 0}{\Psi _{i,r}^{+}z^r}$ for all $i\in I_0$.

\begin{proposition}\label{key-proposition}
		Let $V$ be a representation in $\mathcal{O}_{\mu}$. For each weight space of $V$, there is a non-zero polynomial $P(z)$ so that for any $i\in I_0$, $P(z)(\phi_i^+(z)-\phi_i^-(z))$  and $P(z)x_i^{\pm}(z)$ are  zero on this weight space. The action of $\phi_i^+(z)$ and  $\phi_i^-(z)$ are rational of degree $\alpha_i(\mu)$ on this weight space and coincide as rational operators.
\end{proposition}
	
\begin{proof}
		As the same argument of \citep[Remark~4.4]{Hernandezshfted}, we can see for each $i\in I_0$, the action of $\phi^+_i(z)$ and $\phi^-_i(z)$ are rational on weight space of a representation in $\mo_\mu$ and coincide as rational operators (See also in \cite[Section~3.6]{GTL}). In particular on each weight space $\phi^+_i(z)$ is equal to $\phi^-_{i,\alpha_i{(\mu)}}z^{\alpha_i(\mu)}$ when $z\rightarrow \infty$ which implies the degree is $\alpha_i(\mu)$ as $\phi^+_{i,\alpha_i(\mu)}$ is invertible. The statement for $x_i^+(z)$ is proved as in \cite[Proposition~3.8]{He2}. We also can see it in the prove of Lemma \ref{rationallemma1}.
\end{proof}

As above, every object of $\mathcal O_\mu$
decomposes as a direct sum of its $\ell$-weight spaces, defined with respect
to the generalized eigenvalues of the currents $\phi_i^+(z)$. The simple representations are determined up to isomorphism by their highest $\ell$-weight $\mathbf{\Psi}$ which is rational with $\mathrm{deg}(\Psi_i)=\alpha_{i}(\mu)$. 
	
A representation $V$ in $\mathcal O_\mu$ is called a highest $\ell$-weight
representation of highest $\ell$-weight $\boldsymbol\Psi\in\mathfrak r_\mu$
if it is generated by a nonzero vector $v$ such that
\[
x_{i,m}^+.v=0
\qquad (i\in I_0,\ m\in\mathbb Z),
\]
and
\[
\phi_i^+(z).v=\Psi_i^+(z)v,
\qquad
\phi_i^-(z).v=\Psi_i^-(z)v
\qquad (i\in I_0),
\]
where $\Psi_i^+(z)$ and $\Psi_i^-(z)$ are the expansions of the  rational
function $\Psi_i(z)$ at $z=0$ and $z=\infty$, respectively:
\[
\Psi_i(z)
=
\sum_{m\ge 0}\Psi_{i,m}^+z^m
=
\sum_{m\ge -\alpha_i(\mu)}
\Psi_{i,-m}^-z^{-m}.
\]
	
	\begin{corollary}\label{nakajimacorollary}
		Assume that $\hat{\mathfrak g}=A_2^{(2)}$ and let $\mu\in -\Lambda^+$. Then a simple representation in  the category $\mathcal{O}_\mu$  is simple as a representation of $\mathcal{U}_q(\bs)$.
	\end{corollary}
	
	\begin{proof}
		Consider a simple representation $V$ in the category $\mathcal{O}_\mu$. From the proof of Lemma \ref{rationallemma1}, we can see the action of $x_{i,m}^+$ (resp. $x_{j,m}^-$) for $m\in \ZZ$ are determined by the action of these operators for $m\ge 1$, which are in $\mathcal{U}_q(\bs)$. Hence $V$ is generated by its highest weight vector and has no other primitive vector as a representation of $\mathcal{U}_q(\bs)$ (a primitive vector is a nonzero vector annihilated by all $x_{i,r}^+$ for $i\in I$ and $r\in \mathbb{Z}_{\ge 0}$). Therefore it is simple.
	\end{proof}
\begin{theorem}\label{rational-theorem}
For $\mu\in\Lambda$, the isomorphism classes of simple objects in
$\mathcal O_\mu$ are parametrized by $\mathfrak r_\mu$.
More precisely, for each $\boldsymbol\Psi\in\mathfrak r_\mu$ there exists a
unique simple highest $\ell$-weight module $L(\boldsymbol\Psi)$ in
$\mathcal O_\mu$, and every simple object of $\mathcal O_\mu$ is of this form.
\end{theorem}
	
\begin{proof}
From the above discussion, we know for all $i\in I_0$, a simple representation in the category $\mathcal{O}_{\mu}$ has a highest $\ell$-weight $\mathbf{\Psi}$ satisfying $\mathrm{deg}(\Psi_i)=\alpha_i(\mu)$ and that $\phi^+_i(z)$, $\phi^-_i(z)$ coincide as rational operators on this representation. Therefore it remains to prove that, for every
$\boldsymbol\Psi\in\mathfrak r_\mu$, the corresponding simple highest $\ell$-weight module belongs to $\mathcal O_\mu$.
		
We consider a representation $L$ of highest $\ell$-weight $\mathbf{\Psi}$ with a highest weight vector $v$ (such a representation can be constructed from a Verma module of
highest $\ell$-weight as in \cite{chari-twisted-98} for instance). It suffices to prove its weight space $L_{\omega^\prime}$, $\omega^\prime\in \mathfrak{t}^*$, are finite-dimensional. We only prove for the positive part. Let $\omega=\mathbf{\Psi}(0)\in \mathfrak{t}^*$ be the highest weight of $L$. 

We follow the induction argument of \cite[Section~5, proof of (b)]{chari-95}. Let $\omega'\in\mathfrak t^*$ be a weight of $L(\boldsymbol\Psi)$, and write $\eta=\omega(\omega')^{-1}$.
Denote $\eta=\omega(\omega^\prime)^{-1}$ and $h$ as the height of $\eta$. We prove by induction on $h$ that the weight space $L(\boldsymbol\Psi)_{\omega\eta^{-1}}$ is finite-dimensional. If $h=0$, there is nothing to prove.  Suppose first that $h=1$, say $\eta=\overline{\alpha_j}$. We prove that $L(\boldsymbol\Psi)_{\omega\overline{\alpha_j}^{-1}}$ is finite-dimensional. By the properties of $\mathbf{\Psi}$, there is a non-zero polynomial $P(z)$ such that for any $i\in I_0$, the operator $P(z)(\phi_i^+(z)-\phi_i^-(z))$  is $0$ on $L_\omega$. For $i,j\in I_0$ and $s\in \ZZ$, by the relation  (\ref{xphi-relation}), 
\begin{align*}
x_{i,s}^+(P(z)x_j^-(z))=0
\end{align*}
on $L_\omega$. As $L$ is simple, we obtain $P(z)x_j^-(z)=0$ on $L_\omega$. This implies that $L_{\omega\overline{\alpha_j}^{-1}}$ is finite-dimensional.
		
Now assume $h\ge2$ and that the assertion is known for all smaller heights.
Fix a factorization 
$\eta=\overline{\alpha_{i_1}}\cdots\overline{\alpha_{i_h}}$ .
The weight space $L_{\omega \eta^{-1}}$    is spanned , in view of Proposition \ref{triangular}, by the vectors of the form 
\begin{align}
x_{i_1,k_1}^-x_{i_2,k_2}^-\cdots x_{i_h,k_h}^-v,
\qquad k_1,\ldots,k_h\in\mathbb Z .\label{vector-form}
\end{align}

It is enough to prove that, for fixed $i_1,\ldots,i_h$, these vectors span a finite-dimensional space; denote this space by $V_{i_1\cdots i_h}$. By the induction hypothesis, there exists $M_1\in \NN$ such that, for all $i\in \left\{i_1,\ldots,i_h\right\}$, $L_{\omega \eta^{-1}\overline{\alpha_{i}}}$ is spanned by the vectors of the form 
\begin{align}
x_{j_2,\ell_2}^-x_{j_3,\ell_3}^-\ldots x_{j_h,\ell_h}^-.v,
\end{align}
where $\overline{\alpha_{j_2}} ~\overline{\alpha_{j_3}}\cdots \overline{\alpha_{j_h}}=\eta\overline{\alpha_{i}}^{-1}$ and $\vert l_2\vert$, $\vert l_3\vert$, $\cdots$ , $\vert l_h\vert\le M_1$. 
		
When $M=2$ , it suffices to prove that $V_{i_1\cdots i_h}$ is contained in the space 
\begin{align}
W=\sum_{k_2=-M_1}^{M_1+2}{x_{i_2,k_2}^{-}}\cdot L_{\omega \eta ^{-1}\alpha _{i_2}}+x_{i_1,0}^{-}\cdot L_{\omega \eta ^{-1}\alpha _{i_1}}+x^{-}_{i_1,1}\cdot L_{\omega \eta ^{-1}\alpha _{i_1}},
\end{align}
since $W$ is finite-dimensional by the induction hypothesis.
		
For this, we prove, by induction on $k_1$, that the vector $(4.1)$ lies in $W$ for every $k_2,\cdots,k_h$ (we assume $k_1\ge 0$, the proof for $k_1\le 0$  being essentially the same). The case $k_1=0$ and $k_1=1$ are obvious. For the inductive step, note that we can assume that $\vert k_2\vert$, $\vert k_3\vert,\cdots,$ $\vert k_h\vert\le M_1$.

Using the relation \ref{xx-relation}, any vector  (\ref{vector-form})  can be written as a linear combination of the vectors 
\begin{align}
x_{i_2,k_2}^{-}x_{i_1,k_1}^{-}x_{i_3,k_3}^{-}\cdots x_{i_h,k_h}^{-}. v,\label{type1}\\
x_{i_2,k_2+1}^{-}x_{i_1,k_1-1}^{-}x_{i_3,k_3}^{-}\cdots x_{i_h,k_h}^{-}. v,\label{type2}\\
x_{i_2,k_2+2}^{-}x_{i_1,k_1-2}^{-}x_{i_3,k_3}^{-}\cdots x_{i_h,k_h}^{-}. v,\label{type3}\\
x_{i_1,k_1-1}^{-}x_{i_2,k_2+1}^{-}x_{i_3,k_3}^{-}\cdots x_{i_h,k_h}^{-}. v,\label{type4}\\ 
x_{i_1,k_1-2}^{-}x_{i_2,k_2+2}^{-}x_{i_3,k_3}^{-}\cdots x_{i_h,k_h}^{-}. v.\label{type5}
\end{align} 
But vectors of types (\ref{type1}), (\ref{type2}) and (\ref{type3}) obviously in $W$. For the vector $x_{i_2,k_2+1}^{-}x_{i_3,k_3}^{-}\cdots x_{i_h,k_h}^{-}. v$ and $x_{i_2,k_2+2}^{-}x_{i_3,k_3}^{-}\cdots x_{i_h,k_h}^{-}.v$, they belong to span\{$x_{j_2,\ell_2}^-x_{j_3,\ell_3}^-\ldots x_{j_h,\ell_h}^-.v$\},
where $\overline{\alpha_{j_2}} ~\overline{\alpha_{j_3}}\cdots \overline{\alpha_{j_h}}=\eta\overline{\alpha_{i}}^{-1}$ and $\vert l_2\vert$, $\vert l_3\vert$, $\ldots$ , $\vert l_h\vert\le M_1$. Then those of type (\ref{type4}) and (\ref{type5}) belong to $W$ by the induction hypothesis on $k_1$.
		
When $M=3$, the proof is similar.
		
This completes the proof of the proposition, and hence  that of Theorem  \ref{rational-theorem}.
\end{proof}
	
\begin{example}
For $M=2$, $i\in I_0$ and $a\in \CC^*$. When $i\ne \sigma(i)$, we can construct one-dimensional representation $L(\boldsymbol \Psi^+_{i,a})$, with the action of the $x^\pm_{j,m}$ equal to $0$ for $j\in I_0,m\in \ZZ$, and 
\begin{align*}
    \phi_j^\pm(z)=\Psi_j(z).
\end{align*}
where $\Psi^+_{j,a}=(1-az\delta_{i,j})$.

When $i=\sigma(i)$,  we can construct one-dimensional representation $L(\boldsymbol \Psi^+_{i,a})$, with the action of the $x^\pm_{j,m}$ equal to $0$ for $j\in I_0,m\in \ZZ$, and 
\begin{align*}
    \phi_j^\pm(z)=\Psi_j(z).
\end{align*}
where $\Psi^+_{j,a}=(1-a^2z^2\delta_{i,j})$.
\end{example}

\begin{definition}
For $i\in I_0$ and $a\in\mathbb C^*$, we call
$L(\boldsymbol\Psi_{i,a}^+)$ the positive prefundamental representation and
$L(\boldsymbol\Psi_{i,a}^-)$ the negative prefundamental representation.
\end{definition}
	
We now define the direct sum of abelian categories
\[
\mathcal O^{sh}
=
\bigoplus_{\mu\in\Lambda}\mathcal O_\mu .
\]
Set $\mathfrak r^{sh}
=\bigsqcup_{\mu\in\Lambda}\mathfrak r_\mu$.
By Theorem~\ref{rational-theorem}, the simple objects of
$\mathcal O^{sh}$ are parametrized by
$\mathfrak r^{sh}$.

\subsection{Shift functors}
Shift homomorphisms were introduced in \cite[Section~10(vii)]{FT} as
quantum affine analogues of the morphisms for shifted Yangians constructed
in \cite{FKPRW}. We define their twisted analogue for the types considered
here, excluding \(A_{2n}^{(2)}\).

Let \(\mu\in\Lambda\), \(\mu'\in-\Lambda^+\), and \(a\in\mathbb C^*\). For
\(i\in I_0\), set
$\iota_i=1+(M-1)\delta_{i=\sigma(i)}$
and
$n_i(\mu')=-\frac{\alpha_i(\mu')}{\iota_i}.$ Define a map
\[
\iota_{\mu,\mu',a}:
\mathcal U_q^\mu(\hat{\mathfrak g}^{\sigma})
\longrightarrow
\mathcal U_q^{\mu+\mu'}(\hat{\mathfrak g}^{\sigma})
\]
on Drinfeld currents by
\begin{align}
	x_i^+(z)
	&\longmapsto x_i^+(z),\\
	x_i^-(z)
	&\longmapsto
	\left(1-(az^{-1})^{\iota_i}\right)^{n_i(\mu')}
	x_i^-(z),\\
	\phi_i^\pm(z)
	&\longmapsto
	\left(1-(az)^{\iota_i}\right)^{n_i(\mu')}
	\phi_i^\pm(z).
\end{align}

\begin{theorem}
    The shift map $\iota_{\mu,\mu^{\prime},a}$ is an algebra homomorphism.
\end{theorem}

\begin{proof}
The above assignment is compatible with defining relations among Drinfeld currents.
\end{proof}

\begin{proposition} [\cite{FT}]
    The shift homomorphism $\iota_{\mu,\mu^{\prime},a}$ is injective.
\end{proposition}

\begin{remark}
$(i)$ This is proved explicitly in type $A$ in \cite[Proposition I.4]{FT}, but for all twisted types besides $A^{(2)}_{2n}$, the same argument works without any essential changes. 

$(ii)$ Consequently, for $\mu\in -\Lambda^+$, $\mathcal{U}_q^\mu(\hgs)$ contains a subalgebra isomorphic to $\mathcal{U}_q^0(\hgs)$ and so a subalgebra isomorphic to $\mathcal{U}_q(\bs)$.
\end{remark}

Let \(\mu\in\Lambda\), \(\mu'\in-\Lambda^+\), and \(a\in\mathbb C^*\). By
the homomorphism \(\iota_{\mu,\mu',a}\), we obtain a functor
\[
\mathcal R_{\mu,\mu',a}:
\mathcal O_{\mu+\mu'}
\longrightarrow
\mathcal O_\mu .
\]

From the defining formulas of $\iota_{\mu,\mu^{\prime},a}$, we obtain the following. 

\begin{proposition}\label{proposition-1}
	Let \(L(\boldsymbol\Psi)\) be a simple object of
	\(\mathcal O_{\mu+\mu'}\). Then
	\(\mathcal R_{\mu,\mu',a}(L(\boldsymbol\Psi))\) contains an
	\(\ell\)-highest vector of \(\ell\)-weight \(\boldsymbol\Psi'\), where
	\[
	\Psi_i'(z)
	=
	\Psi_i(z)
	\left(1-(az)^{\iota_i}\right)^{n_i(\mu')}
	\qquad (i\in I_0).
	\]
	Equivalently,
	\[
	\boldsymbol\Psi'
	=
	\boldsymbol\Psi
	\prod_{i\in I_0}
	\left(\boldsymbol\Psi_{i,a}^+\right)^{n_i(\mu')},
	\]
	where \(\boldsymbol\Psi_{i,a}^+\) denotes the one-dimensional
	\(\ell\)-weight whose \(i\)-th component is
	\(1-(az)^{\iota_i}\). In particular,
	\(\mathcal R_{\mu,\mu',a}(L(\boldsymbol\Psi))\) admits
	\(L(\boldsymbol\Psi')\) as a subquotient.
\end{proposition}



\section{Fusion Product}
	
We construct the fusiom product of representations of shifted twisted quantum affine algebras in the category $\mo^{sh}$ by using the deformed Drinfeld coproduct (Theorem \ref{q-character-compality}). We establish a simple module in $\mo^{sh}$ is a quotient of a fusion product of various prefundamental and constant representations (Corollary \ref{subquotient}).
	
\subsection{\texorpdfstring{$q$-characters}{q-characters}}
Following \cite{wang2023QQ,FR}, we define the
$q$-character map
\[
\chi_q:K_0(\mathcal O_\mu)\longrightarrow \mathcal E_{\ell,\mu}.
\]
Here $K_0(\mathcal O_\mu)$ is the Grothendieck group of the abelian category
$\mathcal O_\mu$, and $\mathcal E_{\ell,\mu}$ denotes the group of formal series with coefficients in
$\mathfrak r_\mu$ as in  \cite{wang2023QQ}.

For $V\in\mathcal O_\mu$, the $q$-character is defined by 
\begin{align*}
\chi_q(V)&=\sum_{\boldsymbol{\Psi}\in\mathfrak{r}_\mu}\dim(V_{\boldsymbol{\Psi}})[\boldsymbol{\Psi}],
\end{align*}
where $V_{\boldsymbol\Psi}$ is the $\ell$-weight space of $\ell$-weight $\boldsymbol\Psi$ as above and $[\boldsymbol\Psi]$ is the map $\delta_{\boldsymbol\Psi}$ (which  assigns $1$ to $\boldsymbol\Psi$ and $0$ to all other $\boldsymbol\Psi^\prime$ ).
	
With the same argument as in \cite{Hernandezshfted}, we obtain the following corollary.
\begin{corollary}
The $q$-character morphism $\chi_q$ is injective.
\end{corollary}
	
\subsection{Deformed Drinfeld coproduct}
The Remark \ref{Drinfeld-coproduct-1} can be obviously generalized to the shifted setting.
\begin{lemma}
    For any $\mu_1,\mu_2\in \Lambda$, there is a $\CC$-algebra homomorphism
\[
\Delta_D: \mathcal{U}_q^{\mu_1+\mu_2}(\hat{\mathfrak{g}}^\sigma)\to\mathcal{U}_q^{\mu_1}(\hat{\mathfrak{g}}^\sigma)\widehat{\otimes }\mathcal{U}_q^{\mu_2}(\hat{\mathfrak{g}}^\sigma).
\]
\end{lemma}
    
For $u$ a formal parameter and $\mu_1,\mu_2\in \Lambda$, 
\begin{align*}
		(\mathcal{U}_q^{\mu_1}(\hat{\mathfrak{g}}^\sigma)\otimes\mathcal{U}_q^{\mu_2}(\hat{\mathfrak{g}}^\sigma))((u))
\end{align*} 
is the algebra of formal Laurent series with coefficients in $\mathcal{U}_q^{\mu_1}(\hat{\mathfrak{g}}^\sigma)\otimes\mathcal{U}_q^{\mu_2}(\hat{\mathfrak{g}}^\sigma)$. The deformed Drinfeld coproduct is the map
\begin{align*}
\Delta_u:\mathcal{U}_q^{\mu_1+\mu_2}(\hat{\mathfrak{g}}^\sigma)\to(\mathcal{U}_q^{\mu_1}(\hat{\mathfrak{g}}^\sigma)\otimes\mathcal{U}_q^{\mu_2}(\hat{\mathfrak{g}}^\sigma))((u))
\end{align*}
given on the Drinfeld currents by
\begin{equation}\label{coproduct}
\begin{aligned}
\Delta_{u}(x_{i}^{+}(z))&=x_{i}^{+}(z)\otimes1+\phi_{i}^{-}(z^{-1})\otimes x_{i}^{+}(zu^{-1}),\\\Delta_{u}(x_{i}^{-}(z))&=1\otimes x_{i}^{-}(zu^{-1})+x_{i}^{-}(z)\otimes\phi_{i}^{+}(z^{-1}u),\\
\Delta_u(\phi_i^\pm(z))&=\phi_i^\pm(z)\otimes\phi_i^\pm(zu).
\end{aligned}  
\end{equation}

\begin{proposition}\label{Delta_u}
$\Delta_u$ is an algebra homomorphism. 
\end{proposition}

\begin{proof}
We define a map $\rho_u:
\mathcal U_q^\mu(\widehat{\mathfrak g}^{\,\sigma})
\longrightarrow
\mathcal U_q^\mu(\hat{\mathfrak g}^{\,\sigma})((u))$ 
defined in the Drinfeld generators via
\begin{align*}
    &\rho_u(x_i^+(z))=x_i^+(zu^{-1}),\\
    &\rho_u(x_i^-(z))=x_i^-(zu^{-1}),\\
    &\rho_u(\phi_i^\pm(z))=\phi_i^\pm(zu).
\end{align*}
Firstly, we prove $\rho_u$ is an algebra homomorphism. 

1. For the relation (\ref{phi+x}), we have 
\begin{align*}
\rho _u\left( \phi _{i}^{+}(z)x_{j}^{\pm}(w) \right) =\phi _{i}^{+}(zu)x_{j}^{\pm}(wu^{-1})&=x_{j}^{\pm}(wu^{-1})\phi _{i}^{+}(zu)g_{ij}(zw)^{\pm 1}\\=\rho _u\left( x_{j}^{\pm}(w)\phi _{i}^{+}(z)g_{ij}(zw)^{\pm 1} \right).   
\end{align*}
For the relation (\ref{phi-x}), the argument is similar.

2. For the relation (\ref{x+x-}), we have 
\begin{align*}
&\rho _u\left( \left[ x_{i}^{+}\left( z \right) ,x_{j}^{-}\left( w \right) \right] \right) 
=\left[ x_{i}^{+}(zu^{-1}),x_{j}^{-}(wu^{-1}) \right] \\
&=\frac{1}{N_i(q_{i}-q_{i}^{-1})}\sum_{s=1}^M{\delta _{\sigma ^s\left( i \right) ,j}}\left( \phi _{i}^{+}\left( z^{-1}u \right) \delta \left( \frac{\zeta ^{-s}w}{z} \right) -\phi _{i}^{-}\left( z^{-1}u \right) \delta \left( \frac{\zeta ^{-s}w}{z} \right) \right)\\
&=\rho _u\left( \frac{1}{N_i(q_{i}-q_{i}^{-1})}\sum_{s=1}^M{\delta _{\sigma ^s\left( i \right) ,j}}\left( \phi _{i}^{+}\left( z^{-1} \right) \delta \left( \frac{\zeta ^{-s}w}{z} \right) -\phi _{i}^{-}\left( z^{-1} \right) \delta \left( \frac{\zeta ^{-s}w}{z} \right) \right) \right).
\end{align*}

3. For the relation (\ref{xx-relation}), we have 
\begin{align*}
&\rho _u\left(\prod_{s=1}^M{\left( u_1-\zeta ^sq^{\pm C_{i,\sigma ^s(j)}}u_2 \right) x_{i}^{\pm}(u_1)x_{j}^{\pm}(u_2)} \right)\\
&=\prod_{s=1}^M{\left( u_1-\zeta ^sq^{\pm C_{i,\sigma ^s(j)}}u_2 \right) x_{i}^{\pm}(u_1u^{-1})x_{j}^{\pm}(u_2u^{-1})}
\\
&=\prod_{s=1}^M{\left( u_1q^{\pm C_{i,\sigma ^s(j)}}-\zeta ^su_2 \right)}x_{j}^{\pm}(u_2u^{-1})x_{i}^{\pm}(u_1u^{-1})\\
&=\rho _u\left( \prod_{s=1}^M{\left( u_1q^{\pm C_{i,\sigma ^s(j)}}-\zeta ^su_2 \right) x_{j}^{\pm}(u_2)x_{i}^{\pm}(u_1)} \right).
\end{align*}

4. The relations (\ref{Serre1}), (\ref{Serre2}), (\ref{Serre3}). If $P^\pm_{i,j}(u_1,u_2)=1$, we have 
\begin{align*}
&\rho_u(\operatorname{Sym}_{u_1,u_2}
\Bigl\{
\bigl(
x_j^{\pm}(v)x_i^{\pm}(u_1)x_i^{\pm}(u_2)-(q^{2Md_{i,j}}+q^{-2Md_{i,j}})
x_i^{\pm}(u_1)x_j^{\pm}(v)x_i^{\pm}(u_2)\\
&+x_i^{\pm}(u_1)x_i^{\pm}(u_2)x_j^{\pm}(v)
\bigr)
\Bigr\})\\
&=
\operatorname{Sym}_{u_1,u_2}\Bigl\{\bigl( x_{j}^{\pm}(vu^{-1})x_{i}^{\pm}(u_1u^{-1})x_{i}^{\pm}(u_2u^{-1})\\
&-(q^{2Md_{i,j}}+q^{-2Md_{i,j}})x_{i}^{\pm}(u_1u^{-1})x_{j}^{\pm}(vu^{-1})x_{i}^{\pm}(u_2u^{-1})
+x_{i}^{\pm}(u_1u^{-1})x_{i}^{\pm}(u_2u^{-1})x_{j}^{\pm}(vu^{-1}) \bigr) \Bigr\}=0.
\end{align*}

If $M=2$ and $P^\pm_{i,j}(u_1,u_2)=\frac{u_1^2q^{\pm4}-u_2^2}{u_1q^{\pm2}-u_2}=u_1q^{\pm2}+u_2$, we have
\begin{align*}
&\rho_u(\operatorname{Sym}_{u_1,u_2}
\Bigl\{\left( u_1q^{\pm 2}+u_2 \right) 
\bigl(
x_j^{\pm}(v)x_i^{\pm}(u_1)x_i^{\pm}(u_2)-(q^{2Md_{i,j}}+q^{-2Md_{i,j}})
x_i^{\pm}(u_1)x_j^{\pm}(v)x_i^{\pm}(u_2)\\
&+x_i^{\pm}(u_1)x_i^{\pm}(u_2)x_j^{\pm}(v)
\bigr)
\Bigr\})\\
&=u
\operatorname{Sym}_{u_1,u_2}\Bigl \{ \left( u_1u^{-1}q^{\pm 2}+u_2u^{-1} \right) \bigl( x_{j}^{\pm}(vu^{-1})x_{i}^{\pm}(u_1u^{-1})x_{i}^{\pm}(u_2u^{-1})\\
&-(q^{2Md_{i,j}}+q^{-2Md_{i,j}})x_{i}^{\pm}(u_1u^{-1})x_{j}^{\pm}(vu^{-1})x_{i}^{\pm}(u_2u^{-1})
+x_{i}^{\pm}(u_1u^{-1})x_{i}^{\pm}(u_2u^{-1})x_{j}^{\pm}(vu^{-1}) \bigr) \Bigr\}=0.
\end{align*}

The verification of the relations \eqref{Serre2} and \eqref{Serre3} is
obtained by the same argument, using the corresponding Serre relations for the
currents. The case $M=3$ is treated in exactly the same way. Hence all defining
relations are preserved, and therefore $\rho_u$ is an algebra homomorphism.

Since $\Delta_u=(\mathrm{Id}\otimes \rho_u)\circ \Delta_D$, we have that  $\Delta_u$ is an algebra homomorphism.

\end{proof}
	
\subsection{Fusion product}
Consider $V\in \mo_\mu$, then for $x^\pm_{i,m}(i\in I_0,m\in \ZZ)$, we have the following lemma.

\begin{lemma}\label{rationallemma1}
For the fixed $\omega$, $i\in I_0$ and $r\in \ZZ$. There exsits a finite number of $\mathbb{C}$ -linear opeartors $f_{k,\lambda}^{\pm}: V_{\omega}\rightarrow V$
$\left( k\ge 0,\lambda \in \mathbb{C} ^* \right)$ such that  for all $m\ge 0$ and $v\in V_\omega$
\begin{align*}
	&x_{i,r\pm m}^{\pm}(v)=\sum_{k\geqslant 0,\lambda \in \mathbb{C} ^*}{\lambda ^m}m^kf_{k,\lambda}^{\pm}(v)  ~~~~~\text{if}\,\,\sigma \left( i \right) \ne i;\\
	&x_{i,2r\pm 2m}^{\pm}(v)=\sum_{k\geqslant 0,\lambda \in \mathbb{C} ^*}{\lambda ^m}m^kf_{k,\lambda}^{\pm}(v)~~~\text{if}\,\,\sigma \left( i \right) =i, M=2;\\
	&x_{i,3r\pm 3m}^{\pm}(v)=\sum_{k\geqslant 0,\lambda \in \mathbb{C} ^*}{\lambda ^m}m^kf_{k,\lambda}^{\pm}(v)~~~\text{if}\,\,\sigma \left( i \right) =i, M=3.
\end{align*}
\end{lemma}
	
\begin{proof}
	We firstly consider the case $M=2, \zeta=-1$.
		
		If $\sigma(i)=i$, then $x_{i.m}^\pm=0=h_{i,m}^\pm$ for $2\nmid m$, then we can just consider $x^{\pm}_{i,m}$ where $2\mid m$. Consider the linear map $\Phi :\mathrm{Hom}\left( V_{\omega},V_{\omega \overline{\alpha _i}} \right) \rightarrow \mathrm{Hom}\left( V_{\omega},V_{\omega \overline{\alpha _i}} \right) 
		$  defined by
		
		$$ \Phi \left( p \right)=\frac{1}{\left[ 2C_{i,i}/d_{i} \right] _{q_{i}}}\left( h_{i,2}p-ph_{i,2} \right). 
		$$
		Since  $\mathrm{Hom}\left( V_{\omega},V_{\omega +\alpha _i} \right) $ is finite dimensional, we consider the Jordan decomposition $\Phi =\Phi _1+\Phi _2$ where $\Phi _1\Phi _2=\Phi _2\Phi _1$, $\Phi _1$ is diagonalizable and $\left\lbrace \Phi _2 \right\rbrace ^S=0$ for $S\ge 1$. It follows from formula \ref{hx-relation} that for $2m\ge 0$, $x^+_{i,2r+2m}=\Phi^m(x^{+}_{i,2r})$, and so 
		$$
		x_{i,2r+2m}^+=\sum_{s=0,\ldots,S}
		\begin{bmatrix}m\\s\end{bmatrix}
		(\Phi_2^s\Phi_1^{m-s})(x_{i,2r}^+)=\sum_{\lambda\in\mathbb{C}}\sum_{s=0,\ldots,S}\begin{bmatrix}m\\s\end{bmatrix}\lambda^{m-s}(\Phi_2^s)(d_\lambda),
		$$
		where $\Phi (x_{i,2r}^+)=\sum_{\lambda\in \CC}d_\lambda$ and for $\lambda\in \CC$, $\Phi_1(d_\lambda)=\lambda d_\lambda$. Because for a fixed $s\ge 0$, $\left[\begin{array}{c}m\\s\end{array}\right]$ is a polynomial in $m$, we obtain the desired result. For $m \le 0$, we replace $\Phi$ by
		$$ \Phi \left( p \right)=\frac{1}{\left[-2 C_{i,i}/d_{i} \right] _{q_{i}}}\left( h_{i,-2}p-ph_{i,-2} \right) .
		$$
	
		If $\sigma(i)\ne i$, consider the linear map $\Phi :\mathrm{Hom}\left( V_{\omega},V_{\omega \overline{\alpha _i}} \right) \rightarrow \mathrm{Hom}\left( V_{\omega},V_{\omega \overline{\alpha _i}} \right) 
		$  defined by
		$$
		\Phi \left( p \right) =\frac{1}{-\left[ C_{i,\sigma (i)}/d_{i} \right] _{q_{i}}+\left[ C_{i,i}/d_{i} \right] _{q_{i}}}\left( h_{i,1}p-ph_{i,1} \right).
		$$
		The rest of the proof is similar to the above.
		
		For the case $M=3$, $\zeta=\exp(\rm i\frac{2\pi}{3})$, it only appears in $D^{(3)}_4$. Then we have 
		$[h_{2,3},x_{2,3r}^{\pm}]=\pm \left[ 2 \right]_{q^3} x_{2,3m+3r}^{\pm};
		\left[ h_{1,1},x_{1,r}^{\pm} \right] =\pm \left[ 2 \right] _qx_{1,r+1}^{\pm}.$ The proof of this case is similar to the above.
\end{proof}

Consider the highest $\ell$-weight modules $V_1$, $V_2$ respectively in $\mathcal{O}_{\mu_1}$, $\mathcal{O}_{\mu_2}$. From the definition of $\Delta_{u}$, we obtain a structure of $\mathcal{U}_q^{\mu_1+\mu_2}(\hat{\mathfrak{g}}^\sigma)$-module on the space of the Laurent formal power series with coefficients in $V_1\otimes V_2$:
$$
	(V_1\otimes V_2)((u)).
$$
	
Let $V=V_1\otimes V_2$ and we define the subspace of rational Laurent formal power series $V(u)\subset V((u))$, then we have the following lemmas:

\begin{lemma}\label{rationallemma2}
	The subspace of rational Laurent formal power series $V(u)$ is stable under the $\mathcal{U}_q^{\mu_1+\mu_2}(\hat{\mathfrak{g}}^\sigma)$-action.
\end{lemma}
	
\begin{proof}
It is enough to prove the assertion for $v_1\otimes v_2$, where $v_1$ and $v_2$ are weight vectors of $(V_1)^+_\lambda$ and $(V_2)^+_{\lambda'}$, respectively. We discuss the action of $x_i^+(z)$; the argument for $x_i^-(z)$ is the same.
		
We firstly consider the case $M=2$, $\zeta=-1$. If $\sigma(i)=i$, by Lemma \ref{rationallemma1}, suppose that for all $2m\ge 0$, 
\begin{align*}
	&x_{i,2p+2m}^+(v_1)=\lambda_{1,+}^mm^{k_{1,+}}f_1^+(v_1), \quad x_{i,2p+ 2m}^+(v_2)=\lambda_{2,+}^mm^{k_{2,+}}f_2^+(v_2).
\end{align*}		
If $\alpha_i(\mu_1)< 0$, then we can suppose that for all $2t\ge  -\alpha_i(\mu_1)$, $\phi_{i,-2t}^-(v_1)=\lambda_3^tt^{k_3}f_3(v_1)$. By the formula \ref{coproduct}, we have 
\begin{align*}
	x^{+}_{i,2p+2m}(v_1\otimes v_2)=A+B,
\end{align*}
		where 
		\begin{align*}
			A&= x_{i,2p+2m}^{+}(v_1)\otimes v_2+u^{2p+2m-\alpha_i(\mu_1)}\phi _{i,\alpha _i(\mu _1)}^{-}\left( v_1 \right) \otimes x_{i,2p+2m-\alpha _i(\mu _1)}^{+}\left( v_2 \right)\\
			&=\lambda_{1,+}^mm^{k_{1,+}}f_1^+(v_1)\otimes v_2+u^{2p+2m-\alpha_i(\mu_1)}\phi _{i,\alpha _i(\mu _1)}^{-}\left( v_1 \right)\otimes \lambda^{m-\frac{\alpha_i(\mu_1)}{2}}_{2,+}(m-\frac{\alpha_i(\mu_1)}{2})^{k_2,+}f_2^+\left( v_2 \right)
		\end{align*}
		and 
		\begin{align*}
			B&=\sum_{s>\frac{-\alpha _i\left( \mu _1 \right)}{2}}{u^{2p+2m+2s}\left( \phi _{i,-2s}^{-}\left( v_1 \right) \otimes x_{i,2p+2m+2s}^{+}\left( v_2 \right) \right)}\\
			&=\sum_{s>\frac{-\alpha _i\left( \mu _1 \right)}{2}}u^{2p+2m+2s}(\lambda_3^ss^{k_3}f_3^+(v_1)\otimes \lambda_{2,+}^{m+s}(m+s)^{k_{2,+}}f_2^+(v_2))\\
			&=\sum_{j=0,\cdots,k_{2,+}}\left[ \begin{array}{c}
				k_{2,+}\\
				j\\
			\end{array} \right] \lambda_{2,+}^mu^{2p+2m}m^{k_{2,+}-j}\mathcal{R}_j(u)f_3^+(v_1)\otimes f_2^+(v_2),
		\end{align*}
		where $\mathcal{R}_j(u)=\sum_{s>\frac{-\alpha _i\left( \mu _1 \right)}{2}}u^{2s}\lambda_3^s\lambda_{2,+}^ss^{j+k_3}.$
		As $\mathcal{R}_j(u)\in \CC(u)$, $x^+_{i,2p+2m}(v_1\otimes v_2)$ make sense in $V(u)$.

		When $\alpha_i(\mu_1 )\ge 0$, then we can suppose that for all $2t\ge 1$, $\phi_{i,-2t}^-(v_1)=\lambda_3^tt^{k_3}f_3(v_1)$. Similar to the proof above, we have
		\begin{align*}
			x^{+}_{i,2p+2m}(v_1\otimes v_2)=A+B,
		\end{align*}
		where 
		\begin{align*}
			A&=\lambda _{1,+}^{m}m^{k_{1,+}}f_{1}^{+}\left( v_1 \right) \otimes v_2+\sum_{\frac{-\alpha _i\left( \mu _1 \right)}{2}<s\le 0}{u^{2p+2m+2s}\phi _{i,-2s}^{-}\left( v_1 \right) \otimes \lambda _{2,+}^{m+s}\left( m+s \right) ^{k_{2,+}}f_{2}^{+}\left( v_2 \right)},
			\\
			B&=\sum_{t=0}^{k_{2,+}}{\left[ \begin{array}{c}
				k_{2,+}\\
					t\\
				\end{array} \right] u^{2p+2m}}\lambda _{2,+}^{m}m^{k_{2,+}-t}\mathcal{R} _j\left( u \right) f_3\left( v_1 \right) \otimes f_{2}^{+}\left( v_2 \right),
		\end{align*}
		where $\mathcal{R} _j\left( u \right) =\sum_{s\ge 1}{u^{2s}\lambda _{3}^{s}\lambda _{2,+}^{s}}s^{k_{3}+t}
		$.
		
If $\sigma(i)\ne i$, the remaining proof is similar to the above.
For the case $M=3,\zeta=\exp(\rm i\frac{2\pi}{3})$, the proof is analogous for $M=2,\zeta=-1$.	
\end{proof}

Before we state the main result of this section which is a cyclicity property, we need the following lemma. 
\begin{lemma}[{\cite[Lemma 6.5]{He2}}]\label{2}
Let $V$ be a $\CC(u)$-vector space generated by vectors $(f_{\lambda,k,k^{\prime}})$ where $\lambda\in\mathbb{C}^*,k\geqslant0,k^{\prime}\geqslant0$. We suppose that only a finite number of these vectors are not equal to 0. For $m\geqslant 0$,  we consider $\mu_m=\sum_{\lambda\in\mathbb{C}^*,k\geqslant0,k^{\prime}\geqslant0}\lambda^mm^ku^{-mk^{\prime}}$ $f_{\lambda,k,k^{\prime}}\in V.$ Then $V$ is generated by the $(\mu_m)_{m\geqslant0}$.
\end{lemma}

\begin{theorem}\label{Cyclic}
Suppose that $V\in \mo_{\mu_1}$, $V^\prime\in \mo_{\mu_2} $ are of highest $\ell$-weight with highest $\ell$-weight vectors $v$ and $v^\prime$ respectively. Then $V_1(u)=(V\otimes V^\prime )(u)$ is cyclic for the action of $\mathcal{U}_q^{\mu_1+\mu_2}(\hat{\mathfrak{g}}^\sigma)\otimes\CC(u)$ generated by a tensor product $v\otimes v^\prime$. 
\end{theorem}
	
\begin{proof}
The proof is adapted from the proof of \cite[Theorem~6.2]{He2}. Set
\[
W=
\left(
\mathcal U_q^{\mu_1+\mu_2}(\hat{\mathfrak g}^{\sigma})
\otimes_{\mathbb C}\mathbb C(u)
\right)(v\otimes v')
\subset (V\otimes V')(u).
\]
We prove that \(W=(V\otimes V')(u)\). 

For convenience, we set $V_\omega=V_\omega^+$ for any $\omega\in\mathfrak{t}^*$.  Consider $v\in (V)_\lambda$ and $v^\prime \in (V^\prime)_{\lambda^\prime}$, For $\omega$ a weight of $V_1(u)$, we have $\omega\le \lambda\lambda^\prime$. The map $h$ is defined in Theorem \ref{rational-theorem}. Let us prove induction  on $h(\frac{\lambda \lambda^\prime}{\omega})\ge 0$ that $(V_1(u))_\omega\subset W$. For $h(\frac{\lambda \lambda^\prime}{\omega})=0$, it is clear. In general, let $\omega_1$ and $\omega_2$ be such that $\omega=\omega_1\omega_2$ and $(V_{\omega _1}\otimes V^\prime_{\omega _2})\left( u \right) \subset W$. It suffices to prove that for all $i\in I_0$ and $m\in \ZZ$,
\begin{align*}
x_{i,m}^-(V_{\omega_1})\otimes V'_{\omega_2}\subset W \text{ and } V_{\omega_1}\otimes x_{i,m}^-(V'_{\omega_2})\subset W.
\end{align*}
Let $w\in V_{\omega_1}$ and $w'\in V'_{\omega_2}$. We first treat the case $\sigma(i)\neq i$. Fix \(p\in\mathbb Z\).  we have 
$$
x_{i,p-m}^-\cdot (w\otimes w')=A+B+C,
$$
where 
\begin{align*}
    &A=u^{p-m}\left( w\otimes x_{i,p-m}^{-}w^{\prime} \right), B=\sum _{0\le r\le \alpha_i(\mu_2)}u^r(x_{i,p-m-r}^{-}w\otimes \phi _{i,r}^{+}w^{\prime}), \\&C=\sum_{r>\alpha_i(\mu_2)}{u^r\left( x_{i,p-m-r}^{-}w\otimes \phi _{i,r}^{+}w^\prime \right)}.
\end{align*}
		
By the proof of Lemma \ref{rationallemma1}, we have for all $m\ge 0$, all $w\in V_{\omega_1}$ and $w^\prime\in  V^\prime_{\omega_2}$, \begin{align*}
x_{i,p-m}^-(w)&=\sum_{\lambda\in\mathbb{C}^*,k\geqslant0}\lambda^mm^kf_{\lambda,k}(w),\\x_{i,p-m}^-(w^\prime)&=\sum_{\lambda\in\mathbb{C}^*,k\geqslant0}\lambda^mm^kf_{\lambda,k}^\prime(w^\prime).
\end{align*}
Moreover, for all $r>\alpha_i(\mu_2)$, 
\begin{align*}
\phi_{i,r}^+(w^\prime)=\sum_{\lambda'\in\mathbb{C}^*,k'\geqslant0}\lambda^rr^kg_{\lambda',k'}(w^\prime).
\end{align*}
Only a finite number of $f_{\lambda,k},f_{\lambda,k}^\prime,g_{\lambda,k}$ are  not equal to zero.
		
We have 
\begin{align*}
A&=\sum_{\lambda \in \mathbb{C} ^*,k\ge 0}{\lambda}^mm^ku^{-m}\left( u^pw\otimes f_{\lambda ,k}^{\prime}\left( w^{\prime} \right) \right) ;\\
B&=\sum_{\lambda \in \mathbb{C} ^*,k\ge 0,j=0,\cdots ,k}{\lambda}^mm^{k-j}f_{\lambda ,k}\left( w \right) \otimes \mathcal{R} _{j,\lambda}(w^\prime);\\
C&=\sum_{\lambda \in \mathbb{C} ^*,k\ge 0,j=0,\cdots ,k}{\lambda ^m}m^{k-j}\,\,  f_{\lambda ,k}\left( w \right) \otimes \overline{\mathcal{R} }_{j,\lambda}(w^\prime),     
\end{align*}
where 
\begin{align*}
\mathcal{R} _{j,\lambda}&=\left[ \begin{array}{c}
	k\\
	j\\
\end{array} \right] \sum_{0\le s\le \alpha _i\left( \mu _2 \right)}{u^s}s^j\lambda ^s\phi _{i,s}^{+},\\
\overline{\mathcal{R} }_{j,\lambda}&=\left[ \begin{array}{c}
	k\\
	j\\
\end{array} \right] \sum_{s>\alpha _i\left( \mu _2 \right) ,\lambda ^{\prime}\in \mathbb{C} ^*,k^{\prime}\ge 0}{u^s\lambda ^ss^j\left( \lambda ^{\prime} \right) ^ss^{k^{\prime}}g_{\lambda ^{\prime},k^{\prime}}}.
\end{align*}

Note that $\mathcal{R} _{j,\lambda}$ and $\overline{\mathcal{R} }_{j,\lambda}$ are independent of $m$. As $\phi_{i,r}^{+}$  equals to $\sum_{\lambda'\in\mathbb{C}^*,k'\geqslant0}\lambda^rr^kg_{\lambda',k'}$ in $\mathrm{End}(V^\prime_{\omega_2})$, it follows from Lemma \ref{2} that $g_{\lambda',k'}$ viewed in $\mathrm{End}(V^\prime_{\omega_2})$ is an operator of $\mathcal{U}^{\mu_2}_q(\hat{\mathfrak{\mathfrak{h}}}^{\sigma})$. And so $\overline{\mathcal{R} }_{j,\lambda}$ is also an operator of $\mathcal{U}^{\mu_2}_q(\hat{\mathfrak{\mathfrak{h}}}^{\sigma})$ viewed in $\mathrm{End}(V^\prime_{\omega_2})$.
		
As  $A+B+C\in W$, it follows from Lemma \ref{2} that for all $\lambda\in \CC^{*}$ and $K\ge 0$,
$$
V _{\omega _1}\otimes f_{\lambda ,K}^{\prime}\left(V^\prime _{\omega _2} \right) \subset W
$$
and 
\begin{align}\label{cyclic-contain}
\left( \sum_{\left\{ \left( k,j \right) |0\le j,k-j=K \right\}}{f_{\lambda ,k}\otimes \mathcal{R} _{j,\lambda}+\sum_{\left\{ \left( k,j \right) |0\le j,k-j=K \right\}}{f_{\lambda ,k}\otimes \overline{\mathcal{R} }_{j,\lambda}}} \right) \left( V_{\omega _1}\otimes V_{\omega _2}^{\prime} \right) \subset W.
\end{align}
Thus we have $V_{\omega_1}\otimes x_{i,m}^-(V_{\omega_2}^\prime)\subset W.$
		
It remains to prove $x_{i,m}^-(V_{\omega_1})\otimes V'_{\omega_2}\subset W.$
Fix \(\lambda\) and let \(K_0\) be maximal such that \(f_{\lambda,K_0}\neq0\).
	We prove by descending induction on \(K\) that
	\[
	f_{\lambda,K}(V_{\omega_1})\otimes V'_{\omega_2}\subset W.
	\] Let $w^\prime$ be an $\ell$-weight vector in $V^\prime_{\omega_2}$. For $K=K_0$, by (\ref{cyclic-contain}), we have 
$$
f_{\lambda ,K_0}\left( V_{\omega _1} \right) \otimes \left( \mathcal{R} _{0,\lambda}\left( w^{\prime} \right) +\overline{\mathcal{R} }_{0,\lambda}\left( \left( w^{\prime} \right) \right) \right) \subset W.	
$$		
From the defining formula of  $\mathcal{R} _{0,\lambda}$ and $\overline{\mathcal{R} }_{0,\lambda}$ and \cite[{Lemma 6.6}]{He2}, we obtain $(f_{\lambda,K_0}(V_{\omega_1})\otimes w^\prime)\subset W $. So $\left(f_{\lambda,K_0}(V_{\omega_1})\otimes V_{\omega_2}^{\prime}\right)\subset W$. For $K\le K_0$, $w\in V_{\omega_1}$, $w^\prime \in V^\prime_{\omega_2}$, the vector 
\begin{align*}
	&f_{\lambda ,K}(w)\otimes \left( \mathcal{R} _{0,\lambda}+\overline{\mathcal{R} }_{0,\lambda} \right) (w^\prime )\\
	&=\left( \sum_{\left\{ \left( k,j \right) |0\le j,k-j=K \right\}}{f_{\lambda ,k}(w)}\otimes \left( \mathcal{R} _{j,\lambda}+\overline{\mathcal{R} }_{j,\lambda} \right) (w^\prime ) \right) \\
    &-\sum_{\left\{ \left( k,j \right) ,k-j=K,k>K \right\}}{f_{\lambda ,k}(w)}\otimes \left( \mathcal{R} _{j,\lambda}+\overline{\mathcal{R} }_{j,\lambda} \right) (w^\prime ).\\
	\end{align*} 
		By \ref{cyclic-contain}, the first term is in $W$. By the induction hypothesis, the second term is
		in $W$.  So $f_{\lambda ,K}\left( w \right) \otimes \left( \left( \mathcal{R} _{0,\lambda}+\overline{\mathcal{R} }_{0,\lambda} \right) w^{\prime} \right) \in W$. If, moreover, $w^\prime$ is an $\ell$- weight vector, from the defining formula fo $\mathcal{R}_{0,\lambda} $ and $\overline{\mathcal{R} }_{0,\lambda}$, we can use \cite[{Lemma 6.6}]{He2} and we have $(f_{\lambda,K}(w)\otimes w')\in W$. Then we obtain $\left(f_{\lambda,K}(V_{\omega_1})\otimes V_{\omega_2}^{\prime}\right)\subset W$ and $x_{i,m}^-(V_{\omega_1})\otimes V_{\omega_2}^{\prime}\subset W.$

For $M=2$, $\sigma(i)=i$ or $M=3$, the argument is the similar.  
This proves the theorem.
\end{proof}
	
Let $\mathcal A\subset\mathbb C(u)$ be the subring of rational functions
regular at $u=1$, namely
\[
\mathcal A
=
\{f(u)\in\mathbb C(u)\mid f(u)\text{ is regular at }u=1\}.
\]
An $\mathcal{A}$-form $\tilde{V}\subset V(u)=(V_1\otimes V_2)(u)$ is a $\mathcal{U}_q^{\mu_1+\mu_2}(\hgs)\otimes\mathcal{A}$-submodule generating $V(u)$ as a $\mathbb{C}(u)$-vector space and so that its intersection with any weight space of $V(u)$ is a finitely generated $\mathcal{A}$-module.
	
\begin{proposition}\label{A-form-fusion}
	Suppose that $V_1$ and $V_2$ are highest $\ell$-weight modules in
	$\mathcal O_{\mu_1}$ and $\mathcal O_{\mu_2}$, respectively. Let $v_1$ and
	$v_2$ be their highest $\ell$-weight vectors. Let
	\[
	\widetilde V
	=
	\left(
	\mathcal A\otimes_{\mathbb C}
	\mathcal U_q^{\mu_1+\mu_2}(\hat{\mathfrak g}^{\sigma})
	\right)(v_1\otimes v_2)
	\subset V(u).
	\]
	Then $\widetilde V$ is an $\mathcal A$-form of $V(u)$.
\end{proposition}
	
\begin{proof}
The proof is the same as that of \cite[Lemma~4.8]{He2}
\end{proof}

From the preceding discussion, we obtain the following theorem and corollary which are the same argument in
\cite[Section~5.3]{Hernandezshfted}.

\begin{theorem}\label{q-character-compality}
	Let $V_1\in\mathcal O_{\mu_1}$ and $V_2\in\mathcal O_{\mu_2}$ be highest
	$\ell$-weight modules. Define
	\[
	V_1*V_2
	:=
	\widetilde V/(u-1)\widetilde V .
	\]
	Then $V_1*V_2$ is a highest $\ell$-weight module in
	$\mathcal O_{\mu_1+\mu_2}$. Moreover,
	\[
	\chi_q(V_1*V_2)=\chi_q(V_1)\chi_q(V_2).
	\]
	We call $V_1*V_2$ the fusion product of $V_1$ and $V_2$.
\end{theorem}

\begin{corollary}\label{subquotient}
    A simple module in $\mo^{sh}$ is a subquotient of a fusion product of various prefundamental representations and a simple constant representation.
\end{corollary}

\section{Finite-dimensional representations}
In this section we study finite-dimensional representations in the category
\(\mathcal O^{sh}\). The criterion is formulated in terms of dominant rational $\ell$-weights, with a separate treatment of the special type
$A_{2n}^{(2)}$ (Theorems \ref{finite-dimensional-1} and \ref{finite-dimensional-2}).

We first recall the notation for Drinfeld polynomials in the twisted case. The
finite-dimensional simple representations of twisted quantum affine algebras
were classified by Chari--Pressley \cite{chari-twisted-98}, and the corresponding
$q$-character theory was developed by Hernandez \cite{Hernandeztwisted}. This
notation will be used to formulate the finite-dimensionality criterion for
shifted twisted quantum affine algebras.

For $i\in I_0$ and $a\in \CC^*$, the $\ell$-weights $\tilde{Y}_{i,a} \in \mathfrak{r}_0$ are given by
 \begin{align*}
\left( \tilde{Y}_{i,a} \right) _j\left( z \right) =\begin{cases}
	\frac{1-a^Mz^Mq^{-M}}{1-a^Mz^Mq^M}&if\quad \sigma \left( i \right) =i\quad and\quad j=i,\\
	\frac{1-azq^{-1}}{1-azq}\,\,  &if\quad \sigma \left( i \right) \ne i\quad and\quad j=i,\\
	1  &else.\\
\end{cases}
\end{align*}


The following theorem is an analogue to \cite[Theorem 6.1]{Hernandezshfted} and the proof works word by word.

\begin{theorem}\label{ordinary-theorem}
The simple finite-dimensional representations of $\mathcal{U}_q^0(\hgs)$ are the $L(\boldsymbol\Psi)$ where $\boldsymbol\Psi(z)(\boldsymbol\Psi(0))^{-1}$ is a monomial in the $\tilde{Y}_{i,a}$ for $i\in I_0$, $a\in \CC^*$.

\end{theorem}

\begin{definition}
    An $\ell$-weight $\boldsymbol\Psi$ is called dominant if $\boldsymbol\Psi$ is a monomial in the following $\ell$-weights:
    
    $(i)$ the $Y_{i,a}$'s for $a\in \CC^*$ and $i\in I_0$,
    
    $(ii)$ the $\boldsymbol\Psi_{i,a}$'s for $a\in \CC^*$ and $i\in I_0$,
    
    $(iii)$ the $\gamma$'s for $\gamma\in \mathfrak{t}^*$.
\end{definition}

\subsection{\texorpdfstring{All cases except $A_{2n}^{(2)}$}{All cases except A{2n}{(2)}}}
\begin{proposition}
    For $\mu\in \Lambda$, the algebra $\mathcal{U}_q^\mu(\hgs)$ admits non-zero finite-dimensional representation if and only if $\mu$ is dominant.

    Let $\mu\in \Lambda^+$ be dominant and $\boldsymbol\Psi\in\mathfrak{r}_\mu$ be such that $\boldsymbol\Psi(\boldsymbol\Psi(0))^{-1}$ is a product of various 
    \begin{center}
        $\tilde{Y}_{i,a}$ and $\boldsymbol\Psi_{i,a}$ for $i\in I_0$, $a\in \CC^*$.
    \end{center}
Then $L(\boldsymbol\Psi)$ is a simple finite-dimensional representation of $\mathcal{U}_q^\mu(\hgs)$.
\end{proposition}

\begin{proof}
The proof is similar to \cite[Proposition 6.2]{Hernandezshfted}.
\end{proof}

 \begin{theorem}\label{finite-dimensional-1}
     The simple finite-dimensional representations of shifted twisted  quantum affine algebras are the $L(\boldsymbol\Psi)$ where $\boldsymbol\Psi$ is dominant.
 \end{theorem}

 \begin{proof}
     For $i\in I_0$, denote by $U_i$ the subalgebra of $\mathcal{U}_q^\mu(\hgs)$ generated by the $x^\pm_{i,m}(m\in \ZZ)$, $\phi^\pm_{i,m}(m\in \ZZ)$. As explained in \cite[Section 2.6]{Hernandeztwisted}, we have different cases:
     \begin{itemize}
         \item If $C_{i,\sigma(i)}=2$, we have an algebra isomorphism $U_{q^M}^{\alpha_i(\mu)/M}(\hat{\mathfrak{s} \mathfrak{l} _2})\to U_i$:
         \begin{align*}
             x^+_k\mapsto x^{+}_{i,kM}, \:x^-_k\mapsto x^-_{i,kM},\:\phi^\pm_m\mapsto \phi^{\pm}_{i,mM}.    
         \end{align*}
         
         \item If $C_{i,\sigma(i)}$=0, we have an algebra isomorphism $U_{q}^{\alpha_i(\mu)}(\hat{\mathfrak{s} \mathfrak{l} _2})\to U_i$:
          \begin{align*}
        x^+_k\mapsto x^{+}_{i,k}, \:x^-_k\mapsto x^-_{i,k},\: \phi^\pm_m\mapsto \phi^{\pm}_{i,m}.  
     \end{align*}
     \end{itemize}
    
Then by \cite[Theorem 6.4]{Hernandezshfted}, we have $\boldsymbol\Psi$ has the correct form and $\boldsymbol\Psi$  is dominant.     
 \end{proof}

\subsection{\texorpdfstring{Case $A_{2n}^{(2)}$}{Case A{2n}{(2)}}}
Let $\mo^{sh}_+$ be the subcategory of representations in $\mo^{sh}$ whose simple constituents has a highest $\ell$-weight $\boldsymbol\Psi$ such that the roots and the poles of $\Psi_i(z)$ are of the form $q^r$.

Following \cite{mukhintwisted,Hernandeztwisted}, for $A_2^{(2)}$ case, we have $A_{1,a}=Y_{1,aq}Y_{1,aq^{-1}}Y_{1,-a}^{-1}$. We also extend Nakajima partial odering \cite{nakajima2004quiver} to $\ell$-weights of $\mathcal{U}_q^\mu(A_2^{(2)})$: we set $\boldsymbol\Psi  \prec \boldsymbol\Psi^\prime$ if $\Psi^\prime\Psi$ is a monomial in the $A_{1,a}$. 

\begin{theorem}
    If $\mathcal{U}_q^\mu(\hgs)=\mathcal{U}_q^\mu(A_2^{(2)})$ and for $\boldsymbol\Psi$ an $\ell$-weight, we have 
    \begin{align*}
        \chi_q(L(\boldsymbol\Psi)) \in [\boldsymbol\Psi]+\sum_{\boldsymbol\Psi^\prime  \prec \boldsymbol\Psi}\NN[\boldsymbol\Psi^\prime].
    \end{align*}
\end{theorem}

\begin{proof}
By Corollary \ref{subquotient}, it suffices to prove the statement for prefundamental representations. For positive prefundamental representations, the assertion is immediate, since
they are one-dimensional. For negative prefundamental representations, the
claim follows from Corollary~\ref{nakajimacorollary}: their $q$-characters
coincide with the $q$-characters of the corresponding negative prefundamental
representations of $\mathcal U_q(\bs)$, which were proved in
\cite{wang2023QQ}.
\end{proof}


\begin{theorem}{\label{finite-dimensional-2}}
    For $\mu\in \Lambda^+$, the simple finite-dimensional representations of shifted twisted quantum affine algebra $\mathcal{U}_q^\mu(A_{2n}^{(2)})$ in $\mo^{sh}_+$ are the $L(\boldsymbol\Psi)$ where $\boldsymbol\Psi$ is dominant.
\end{theorem}

\begin{proof}
It suffices to prove for $\hgs=A_2^{(2)}$.
Let $\mu\in \Lambda^+$ and suppose that $L(\boldsymbol\Psi)$ is a simple finite-dimensional representation of $\mathcal{U}_q^\mu(\hgs)$. Thus $\boldsymbol\Psi(z)$ is a rational fraction of non-negative degree.
    Without loss of generality, we may assume $\Psi(0)=1$. There is a non-unique factorization $\boldsymbol\Psi=\boldsymbol\Psi_0\boldsymbol\Psi_+$ where 
\begin{align*}
        &\boldsymbol\Psi_0=(\boldsymbol\Psi_{a_1}\boldsymbol\Psi_{b_1}^{-1})\cdots (\boldsymbol\Psi_{a_T}\boldsymbol\Psi_{b_T}^{-1})\quad \text{for cetain} \quad T\ge 0,\: a_1,\ldots,a_T,\:b_1,\ldots,b_T\in q^{\ZZ},\\
        &\boldsymbol\Psi_+=\boldsymbol\Psi_{c_1}\cdots\boldsymbol\Psi_{c_M}\quad \text{where}\quad M=\deg(\boldsymbol\Psi)\ge 0 \quad \text{and}\quad c_1,\ldots,c_M\in q^{\ZZ}.
\end{align*}

We will denote $\mathcal{F}=\{1\le j\le T\mid a_j\in b_jq^{-2\NN}$ and $ \mathcal{J}=\{1\le j\le T\mid a_j\notin b_jq^{-2\NN}\} $ so that for $j\in \mathcal{F}$, $L(\boldsymbol\Psi_{a_j}\boldsymbol\Psi_{b_j}^{-1})$ is finite-dimensional and for $j\notin \mathcal{J}$, $L(\boldsymbol\Psi_{a_j}\boldsymbol\Psi_{b_j}^{-1})$ is infinite-dimensional.

Moreover, we also assume that for any $1\le j\le M$, $1\le j^\prime\le N$:
\begin{align}{\label{finite-condition}}
    (c_j\notin b_{j^\prime}q^{-2\NN}\,\, \text{if} \,\,j^\prime\in \mathcal{J})\,\, \text{and} \,\, (c_j\notin \{b_{j^\prime}, b_{j^\prime}q^{-2}, \ldots, a_{j^\prime}q^2 \}\,\, \text{if}\,\, j^{\prime}\in \mathcal{F}).
\end{align}

First $L(\boldsymbol\Psi)$ is a subquotient of $L\left( \boldsymbol\Psi _0 \right) *L\left( \boldsymbol\Psi _+ \right) $, so the multiplicities in $\chi_q(L(\boldsymbol\Psi))$
 are lower or equal to the 
multiplicities in $\chi_q(L\left(\boldsymbol \Psi _0 \right)[\boldsymbol\Psi_+])$. Assume that for $j\in \mathcal{F}$, $ b_j=a_jq^{2T_j}$. By \cite[Appedix B]{wang2023QQ} and \cite[Theorem 4.9]{Fujita-Qin}, we have 

\begin{align*}
\chi_q(L((\boldsymbol\Psi_+)^{-1})) &\in [\boldsymbol\Psi_+]^{-1} \prod \left( 1 + \sum_{j,1 \le m} \left( A_{1,c_j}^{-1} A_{1,c_j q^{-2}}^{-1} \cdots A_{1,c_j q^{-2m+2}}^{-1} \right) \right. \\
&\quad + \left. \sum_{j,1 \le k \le m} \left( A_{1,c_j}^{-1} A_{1,c_j q^{-2}}^{-1} \cdots A_{1,c_j q^{-2m+2}}^{-1} \, A_{1,-c_jq}^{-1} A_{1,-c_jq^{-1}}^{-1} \cdots A_{1,-c_j q^{-2k+3}}^{-1} \right) \right),\\
\chi_q\left( L\left( \boldsymbol\Psi_0 \right) \right) &\in \left[ \boldsymbol\Psi_0 \right]   
\prod(1+\sum_{1\le m\le T_j}{\left( A_{1,b_j}^{-1}A_{1,b_jq^{-2}}^{-1}\cdots A_{1,b_jq^{-2m+2}}^{-1} \right)}
\\&+\sum_{1\le k\le m\le T_j}{\left( A_{1,b_j}^{-1}A_{1,b_jq^{-2}}^{-1}\cdots A_{1,b_jq^{-2m+2}}^{-1}A_{1,-b_jq}^{-1}\cdots A_{1,-b_jq^{-2k+3}}^{-1} \right)})_{j \in \mathcal{F}} 
\\
&\quad \times \prod \left\{ \left( 1 + \sum_{1 \le m} \left( A_{1,b_j}^{-1} A_{1,b_j q^{-2}}^{-1} \cdots A_{1,b_j q^{-2m+2}}^{-1} \right) \right) \right. \\
&\quad + \left. \sum_{1 \le k \le m} \left( A_{1,b_j}^{-1} A_{1,b_j q^{-2}}^{-1} \cdots A_{1,b_j q^{-2m+2}}^{-1} A_{1,-b_jq}^{-1} A_{1,-b_j q^{-1}}^{-1} \cdots A_{1,-b_j q^{-2k+3}}^{-1} \right) \right\}_{j \in \mathcal{J}},
\end{align*}

From (\ref{finite-condition}), only the $
\ell$-weight $[\boldsymbol\Psi_+]^{-1}$ contributes to an $\ell$-weight of $L(\boldsymbol\Psi_0)$. Then we obtain that the multiplicities of 
$\ell$-weight in $\chi_q(L(\boldsymbol\Psi_0))$
 are lower or equal to the 
multiplicities in $(\chi_q(L\left(\boldsymbol \Psi  \right)))\cdot[\boldsymbol\Psi_+]^{-1}$. Thus we obtain the isomorphism $L(\boldsymbol\Psi)\simeq L(\boldsymbol\Psi_0)*L(\boldsymbol\Psi_+)$. This implies that $L\left(\boldsymbol \Psi _0 \right)$ is finite-dimensional. By Theorem \ref{ordinary-theorem}, we obtain the desired result.   
\end{proof}

\section{\texorpdfstring{Restriction representations and $q$-character relations}{Restriction representations and q-character relations}}

In this section we construct a restriction representations from the representations of twisted quantum affine Borel algebras to the representations of shifted twisted quantum affine algebras (Proposition \ref{restriction}). Combined with the finite-dimensional classification of
Section~6, we establish a $q$-characters formula for simple finite-dimensional representations of  shifted twisted quantum affine algebras in terms of the $q$-characters of the corresponding simple representations of the twisted quantum affine Borel algebra $\U_q(\bs)$ (Theorem \ref{q-character-shifted-borel}).






\subsection{\texorpdfstring{Restriction representations and $q$-character relations}{Restriction representations and q-character relations}}
Let $\mu\in\Lambda$ and for $V\in \mo_{\bs}$, we consider the subspace $V_\mu$ of vectors $v\in V$ so that for any $i\in I_0$, $\phi_i(z). v\in V(z)$ has a degree lower or equal to $\alpha_i(\mu)$. Set $V_{<\mu}=\sum_{i\in I_0}V_{\mu-\omega_i^\vee}$.

\begin{proposition}
    For $v\in V_{\mu}$ (resp. $V_{<\mu}$), there is a $T>0$ so that $\phi^+_{j,s}. v$, $x^+_{j,m}. v$, $x^-_{j,m}. v$ are in $V_{\mu}$ (resp. $V_{<\mu}$) for any $j\in I_0$, $s\ge 0$, $m\ge T$. 
\end{proposition}

\begin{proof}
It suffices to show the statement for $V_\mu$.
We first treat the case of $x_{j,m}^+v$. For $v\in V_\mu$, and fix $i,j\in I_0$. Set
\begin{align*}
    \alpha:=\alpha_{ i}(\mu),
\qquad
d_m:=\deg\bigl(\phi_i^+(z)x_{j,m}^+.v\bigr).
\end{align*}
We shall prove that $d_m\le \alpha$ for all sufficiently large $m$.

Firstly, we prove the case $M=2$. Then we have 
\begin{align*}
\phi_i^+(z)\Bigl(
x_{j,m-2}^+.v
+a z\,x_{j,m-1}^+.v
-b z^2x_{j,m}^+.v
\Bigr)
=
\Bigl(
b\,x_{j,m-2}^+
-a z\,x_{j,m-1}^+
-z^2x_{j,m}^+
\Bigr)\phi_i^+(z).v,
\end{align*}
where 
\begin{align*}
a=q^{C_{i,\sigma(j)}}-q^{C_{i,j}},
\qquad
b=q^{C_{i,\sigma(j)}+C_{i,j}}.
\end{align*}
Now observe that the three terms on the left-hand side have degrees 
\begin{align*}
d_{m-2},\qquad d_{m-1}+1,\qquad d_m+2.
\end{align*}
On the other hand, since $v\in V_\mu$, we have $    \deg\bigl(\phi_i^+(z)v\bigr)\le \alpha$,
hence the degree of the right-hand side is at most $\alpha+2$.

Assume that there are infinitely many integers $m$ such that $d_m>\alpha$. Call such $m$ a bad index. We first claim that if $m$ is a bad index, then every $m$ satisfies 
\begin{align*}
    d_m\le \max\{d_{m-1}-1,\ d_{m-2}-2\}.
\end{align*}

Otherwise, we have $d_m>d_{m-1}-1$ and $
d_m>d_{m-2}-2$. Then the term $-bz^2\phi_i^+(z)x_{j,m}$ has strictly largest degree, namely $d_m+2$. We obtain $d_m+2\le \alpha+2$, hence  $d_m\le \alpha$, contradicting the assumption of $m$. 

Now let $m_0$ be a bad index. Either $d_{m_0-1}\ge d_{m_0}+1$ or $d_{m_0-2}\ge d_{m_0}+2.$ In particular, there exists $m_1\in\{m_0-1,m_0-2\}$ such that $m_1<m_0$, $
d_{m_1}\ge d_{m_0}+1$, and therefore 
$m_1$ is again a bad index. Repeating the same argument, we construct a strictly decreasing sequence of bad indices 
\begin{align*}
    m_0>m_1>m_2>\cdots>m_r>\cdots
\end{align*}
such that 
$d_{m_t}\ge d_{m_{t-1}}+1
$ for every $t\ge 1.$ Contradiction as weight spaces are finite dimensional.

The case of $x_{j,m}^-v$ is analogous. And it is clear for the $\phi _{j,s}^{+}.v$ which commutes with $\phi _{i}^{+}\left( z \right) $.

\end{proof}

We regard the elements of 
$\mathcal U_q(\mathfrak b^{\sigma})$ as the corresponding
operators of $V$. Let $v\in V_\mu$, and let $i\in I_0$.
We use the convention $x_i^+(u)=\sum_{m\in\mathbb Z}x_{i,m}^+ u^{-m}.$
Set \(t=u^{-1}\). Then the above current can be written as
$x_i^\pm(u)=\sum_{m\in\mathbb Z}x_{i,m}^\pm t^m .$

Choose $T$ as in the preceding proposition. By the rationality result,
the series
\[
        \sum_{m\ge T}x_{i,m}^{+}.v\,u^{-m}
        =
        \sum_{m\ge T}x_{i,m}^{+}.v\,t^m
        \in V_\mu((t))
\]
is the expansion at $t=0$, equivalently at $u=\infty$, of a
$V_\mu$-valued rational function in the variable $t$. Let $d$ be its
degree as a rational function of $t$. Equivalently, $d$ is the order of
its pole at \(t=\infty\), i.e. at \(u=0\). Choose an integer $T^\prime>\max\{T,d\}$.
Then
\[
    \sum_{m>T^\prime}x_{i,m}^{+}.v\,u^{-m}
        =
        \sum_{m>T^\prime}x_{i,m}^{+}.v\,t^m
\]
is again the expansion at \(t=0\) of a \(V_\mu\)-valued rational function in
\(t\). Indeed,
\[
        \sum_{m>T^\prime}x_{i,m}^{+}.v\,t^m
        =
        \sum_{m\ge T}x_{i,m}^{+}.v\,t^m
        -
        \sum_{T\le m\le T'}x_{i,m}^{+}.v\,t^m .
\]
The first term has degree $d<T^\prime$, while the second term is a polynomial in
$t$ of degree at most $T^\prime$. Hence the rational function on the left has
degree at most $T^\prime$. Therefore its expansion at $t=\infty$ is of the
form $-\sum_{m\le T^\prime}v_m^{(T^\prime)}t^m.$
Equivalently, in the variable $u$, we have $\sum_{m>T^\prime}x_{i,m}^{+}.v\,u^{-m}=-\sum_{m\le T^\prime}v_m^{(T^\prime)}u^{-m}.$

For $m>T^\prime$, we set $v_m^{(T^\prime)}=x_{i,m}^{+}.v$.
Thus we obtain a family
$\bigl(v_m^{(T^\prime)}\bigr)_{m\in\mathbb Z}$.
We now show that this family is independent of the choice of $T^\prime$. Let
$T^{\prime\prime}>T^\prime$. Then
\[
\begin{aligned}
-\sum_{m\le T^{\prime\prime}}v_m^{(T^{\prime\prime)}}u^{-m}
&=
\sum_{m>T^{\prime\prime}}v_m^{(T^{\prime\prime)}}u^{-m} =
\sum_{m>T^\prime}v_m^{(T^\prime)}u^{-m}
-
\sum_{T^{\prime\prime}\ge m > T^\prime}v_m^{(T^\prime)}u^{-m}\\
&=
-\sum_{m\le T^\prime}v_m^{(T^\prime)}u^{-m}
-
\sum_{T^{\prime\prime}\ge m>T^\prime}v_m^{(T^\prime)}u^{-m} 
=-\sum_{m\le T^{\prime\prime}}v_m^{(T^\prime)}u^{-m}.
\end{aligned}
\]
Comparing the expansions, we obtain $v_m^{(T^{\prime\prime})}=v_m^{(T^{\prime})}$ for $m\le T''$.
For $m>T^{\prime\prime}$, both families are equal to $x_{i,m}^{+}.v$. Hence $v_m^{(T^{\prime\prime})}=v_m^{(T^\prime)}
$ for all $m\in\mathbb Z$ .

Thus the following definition is independent of the auxiliary integer
$T^\prime$: $\widetilde{x}_{i,m}^{+}.v:=v_m^{(T')}$, for all $m\in \ZZ$. Applying the same construction to the negative current, we obtain operators $\widetilde{x}_{i,m}^{-}$ ($i\in I_0,m\in \ZZ$) on $V_\mu$. These operators are well defined on the quotient $\widetilde V_\mu := V_\mu/V_{<\mu}.$

Moreover, $\phi_i^+(u)$  is rational on $\widetilde V_\mu$
with degree $\alpha_i(\mu)$.
Expanding it at $u=\infty$, we obtain an operator-valued Laurent series $  \phi_i^-(u)
        \in
        u^{\alpha_i(\mu)}
        \operatorname{End}(\widetilde V_\mu)[[u^{-1}]].$

\begin{proposition}\label{restriction}
The operators $\widetilde{x}_{i,m}^{\pm},$ $\phi_i^{\pm}(u)$ $(i\in I_0,m\in \ZZ)$
constructed above on $\widetilde V_\mu$ endow $\widetilde V_\mu$ with a
structure of a $\mathcal U_q^\mu(\hat{\mathfrak g}^{\sigma})$-module.
\end{proposition}

\begin{proof}
We have to verify that the defining relations of the shifted twisted quantum
affine algebra
\(\mathcal U_q^\mu(\hat{\mathfrak g}^{\,\sigma})\)
are satisfied by the operators constructed above. 

We choose a vector $v$ and $T$ sufficiently large so that, for every
$i\in I_0$ and every $m\ge T$, the operators constructed above
coincide with the original ones on $v$, namely $        \widetilde{x}_{i,m}^{\pm}.v_0 =
x_{i,m}^{\pm}.v_0.$

We will write the truncated version of the relations with a positive  part under the case $M=2$. For the negative part and the case $M=3$,  the argument is similar.

Let 
\[
X_i(z):=x_{i}^{+,T}(z)=\sum_{r\ge T}{x_{i,r}^{+}z^{-r}}.
\]
\begin{itemize}
    \item For the relation 
\[
\prod_{r=1}^2{\left( 1-\zeta ^rq^{C_{i,\sigma ^r\left( j \right)}}zw \right)}\phi _{i}^{+}(z)x_{j}^{+}(w)-\prod_{r=1}^2{\left( q^{C_{i,\sigma ^r\left( j \right)}}-\zeta ^rzw \right)}x_{j}^{+}(w)\phi _{i}^{+}(z)=0,
\] the truncated version is 
\begin{align*}
&\prod_{r=1}^2{\left( 1-\zeta ^rq^{C_{i,\sigma ^r\left( j \right)}}zw \right)}\phi _{i}^{+}(z)X_j(w)-\prod_{r=1}^2{\left( q^{C_{i,\sigma ^r\left( j \right)}}-\zeta ^rzw \right)}X_j(w)\phi _{i}^{+}(z)\\
&=w^{1-T}\Bigl[ (a-b)z\,\phi _{i}^{+}(z)x_{j,T}^{+}-(b-a)z\,x_{j,T}^{+}\phi _{i}^{+}(z)-ab\,z^2\phi _{i}^{+}(z)x_{j,T+1}^{+}+z^2x_{j,T+1}^{+}\phi _{i}^{+}(z) \Bigr] \\
&+w^{2-T}\Bigl[ -ab\,z^2\phi _{i}^{+}(z)x_{j,T}^{+}+z^2x_{j,T}^{+}\phi _{i}^{+}(z) \Bigr],
\end{align*}
where $a=q^{C_{i,\sigma(j)}}, b=q^{C_{i,\sigma^2(j)}}.$
\item  For the relation
\[
\prod_{r=1}^2{(u_1}-\zeta ^rq^{C_{i,\sigma ^r(j)}}u_2)x_{i}^{+}(u_1)x_{j}^{+}(u_2)-\prod_{r=1}^2{(u_1q^{C_{i,\sigma ^r(j)}}}-\zeta ^ru_2)x_{j}^{+}(u_2)x_{i}^{+}(u_1)=0,
\]
the truncated version is 
\begin{align*}
    &\prod_{r=1}^2{(u_1}-\zeta ^rq^{C_{i,\sigma ^r(j)}}u_2)X_i(u_1)X_j(u_2)-\prod_{r=1}^2{(u_1q^{C_{i,\sigma ^r(j)}}}-\zeta ^ru_2)X_j(u_2)X_i(u_1)\\&=u_{1}^{2-T}x_{i,T}^{+}X_j(u_2)+u_{1}^{1-T}x_{i,T+1}^{+}X_j(u_2)\\
	&\quad -ab\,u_{1}^{2-T}X_j(u_2)x_{i,T}^{+}-ab\,u_{1}^{1-T}X_j(u_2)x_{i,T+1}^{+}\\
	&\quad +(a-b)u_{1}^{1-T}u_2\,x_{i,T}^{+}X_j(u_2)+(a-b)u_{1}^{1-T}u_2\,X_j(u_2)x_{i,T}^{+}\\
	&\quad +(a-b)u_1u_{2}^{1-T}X_i(u_1)x_{j,T}^{+}+(a-b)u_1u_{2}^{1-T}x_{j,T}^{+}X_i(u_1)\\
	&\quad -(a-b)u_{1}^{1-T}u_{2}^{1-T}x_{i,T}^{+}x_{j,T}^{+}-(a-b)u_{1}^{1-T}u_{2}^{1-T}x_{j,T}^{+}x_{i,T}^{+}\\
	&\quad -ab\,u_{2}^{2-T}X_i(u_1)x_{j,T}^{+}-ab\,u_{2}^{1-T}X_i(u_1)x_{j,T+1}^{+}\\
	&\quad +u_{2}^{2-T}x_{j,T}^{+}X_i(u_1)+u_{2}^{1-T}x_{j,T+1}^{+}X_i(u_1),
\end{align*}
where $a=q^{C_{i,\sigma(j)}}, b=q^{C_{i,\sigma^2(j)}}.$
\item If $C_{i,j}=-1$, the relation is
\begin{align*}
    \operatorname{Sym}_{u_1,u_2}\{P_{ij}^{+}(u_1,u_2)&(x_j^{+}(v)x_i^{+}(u_1)x_i^{+}(u_2)\\-(q^{2Md_{ij}}+q^{-2Md_{ij}})&x_i^{+}(u_1)x_j^{+}(v)x_i^{+}(u_2)+x_i^{+}(u_1)x_i^{+}(u_2)x_j^{+}(v))\}=0.
\end{align*}The truncated version is 
\begin{align*}
    &\operatorname{Sym}_{u_1,u_2}\big\{ P_{ij}^{+}(u_1,u_2)\big[X_j(v)X_i(u_1)X_i(u_2) - (q^{2Md_{i,j}}+q^{-2Md_{i,j}})X_i(u_1)X_j(v)X_i(u_2) \\
    &\qquad + X_i(u_1)X_i(u_2)X_j(v)\big] \big\} \\
    &= \begin{cases}
        0,&\text{if } C_{i,\sigma (i)}=2,\ d_{i,j}=\frac{1}{2},\ P_{ij}^{+}(u_1,u_2)=1 ,\\
        0,&\text{if } C_{i,\sigma (i)}=0,\ \sigma (j) \ne j,\ d_{i,j}=\frac{1}{4},\ P_{ij}^{+}(u_1,u_2)=1, \\
        R_{i,j}& \text{if } C_{i,\sigma (i)}=0,\ \sigma (j)=j,\ d_{i,j}=\frac{1}{2},\ P_{ij}^{+}(u_1,u_2)=\frac{u_{1}^{2}q^{ 4}-u_{2}^{2}}{u_1q^{ 2}-u_2},
    \end{cases}
\end{align*}
where
\begin{align*}
R_{i,j}=&\operatorname{Sym}_{u_1,u_2}\Bigl\{ u_{1}^{1-T}\Bigl[ q^2\bigl( X_j(v)x_{i,T}^{+}X_i(u_2)-(q^2+q^{-2})x_{i,T}^{+}X_j(v)X_i(u_2)+\\&x_{i,T}^{+}X_i(u_2)X_j(v) \bigr)+
\bigl( X_j(v)X_i(u_2)x_{i,T}^{+}\\&-(q^2+q^{-2})X_i(u_2)X_j(v)x_{i,T}^{+}+X_i(u_2)x_{i,T}^{+}X_j(v) \bigr) \Bigr] \Bigr\}.
\end{align*}
\item If $C_{i,\sigma(i)=1}$, the relation is 
\begin{align*}
    &\operatorname{Sym}_{u_1,u_2,u_3}\{[q^{3/2}u_{1}^{-1}-(q^{1/2}+q^{-1/2})u_{2}^{-1}+q^{-3/2}u_{3}^{-1}]x_{i}^{+}(u_1)x_{i}^{+}(u_2)x_{i}^{+}(u_3)\}=0,\\
    &\operatorname{Sym}_{u_1,u_2,u_3}\{[q^{-3/2}u_{1}^{+1}-(q^{1/2}+q^{-1/2})u_{2}^{+1}+q^{3/2}u_{3}^{+1}]x_{i}^{+}(u_1)x_{i}^{+}(u_2)x_{i}^{+}(u_3)\}=0. 
\end{align*}
The truncated version is 
\begin{align*}
    &\operatorname{Sym}_{u_1,u_2,u_3}\{[q^{3/2}u_{1}^{-1}-(q^{1/2}+q^{-1/2})u_{2}^{-1}+q^{-3/2}u_{3}^{-1}]X_i(u_1)X_i(u_2)X_i(u_3)\}\\
    &=-\operatorname{Sym}_{u_1,u_2,u_3}\{ q^{3/2}u_{1}^{-T}x_{i,T-1}^{+}X_i(u_2)X_i(u_3)-(q^{1/2}+q^{-1/2})u_{2}^{-T}X_i(u_1)x_{i,T-1}^{+}X_i(u_3)\\ 
    &+q^{-3/2}u_{3}^{-T}X_i(u_1)X_i(u_2)x_{i,T-1}^{+} \};\\
    &\operatorname{Sym}_{u_1,u_2,u_3}\{[q^{-3/2}u_{1}^{+1}-(q^{1/2}+q^{-1/2})u_{2}^{+1}+q^{3/2}u_{3}^{+1}]x_{i}^{+}(u_1)x_{i}^{+}(u_2)x_{i}^{+}(u_3)\}\\
    &=\operatorname{Sym}_{u_1,u_2,u_3}\{ q^{-3/2}u_{1}^{1-T}x_{i,T}^{+}X_i(u_2)X_i(u_3)-(q^{1/2}+q^{-1/2})u_{2}^{1-T}X_i(u_1)x_{i,T}^{+}X_i(u_3)\\
    &+q^{3/2}u_{3}^{1-T}X_i(u_1)X_i(u_2)x_{i,T}^{+} \}.
\end{align*}
\end{itemize}

Notice if $i\in I_0$, then 
$[x_{i, rN_i}^{+}, x_{i, r^{\prime}N_i}^{-}]$ equals to $\dfrac{\phi_{i, (r+r^{\prime})N_i}^{+}-\phi_{i, (r+r^{\prime})N_i}^{-}}{q_{i}-q_{i}^{-1}}.$ Hence the verification of the relation involving $[x_i^+(z),x_i^-(w)]$ is
parallel to the proof of \cite[Proposition~7.7]{Hernandezshfted}.

Using the same strategy as in \cite[Proposition~7.7]{Hernandezshfted}, together
with the truncated relations established above, we obtain that the operators $\widetilde{x}_{i,m}^{\pm}$, $\phi_i^{\pm}(u)$ $(i\in I_0,m\in \ZZ)$
satisfy the defining relations of $\mathcal U_q^\mu(\hat{\mathfrak g}^{\sigma})$.
\end{proof}

For $i\in I_0$, let $\chi_i$ be the character of $L^{\bs}(\boldsymbol\Psi_{i,1})$. It is proved in \cite[Theorem 4.24]{wang2023QQ} that
\begin{align}\label{q-character-identity}
    \chi_q(L^{\bs}(\boldsymbol\Psi_{i,a}))=[\boldsymbol\Psi_{i,a}]\cdot \chi_i.
\end{align}

We keep the notation and the scope of the finite-dimensional classification in
Section~6. Thus, outside type $A_{2n}^{(2)}$, $\mathbf\Psi$ denotes a dominant
highest $\ell$-weight in the classification of simple finite-dimensional
$\mathcal U_q^\mu(\hgs)$-modules. In type $A_{2n}^{(2)}$, $\mathbf\Psi$ is
restricted to the category $\mo^{sh}_+$.

\begin{theorem}
\label{q-character-shifted-borel}
Let $L(\mathbf\Psi)$ be a simple finite-dimensional
$\mathcal U_q^\mu(\hgs)$-module covered by the classification in Section~6,
and let $L^{\bs}(\mathbf\Psi)$ be the corresponding simple
$\mathcal U_q(\bs)$-module with highest $\ell$-weight $\mathbf\Psi$.
Then
\[
\chi_q\bigl(L^{\bs}(\mathbf\Psi)\bigr)
=
\chi_q\bigl(L(\mathbf\Psi)\bigr)
\prod_{i\in I_0}\chi_i^{\alpha_i(\mu)/\iota_i}.
\]
\end{theorem}

\begin{proof}
By the finite-dimensional classification in Section~6, under the hypotheses of the theorem, the highest \(\ell\)-weight \(\boldsymbol\Psi\) admits a factorization
\[
\boldsymbol\Psi=\boldsymbol\Psi_0\boldsymbol\Psi_+,
\]
where \(L(\boldsymbol\Psi_0)\) is a finite-dimensional simple module for $\U_q(\lgs)$ and $\boldsymbol\Psi_+$
is a product of positive prefundamental $\ell$-weights. 

We regard \(L(\boldsymbol\Psi_0)\) as a
\(\mathcal U_q(\lgs)\)-module and consider
\[
V=
L(\boldsymbol\Psi_0)\otimes L^{\bs}(\boldsymbol\Psi_+).
\]
Let $v_0$ and $v_+$ be the highest $\ell$-weight vectors of
$L(\boldsymbol\Psi_0)$ and $L^{\bs}(\boldsymbol\Psi_+)$, respectively.  By
Lemma~\ref{Drinfeld-coproduct-module}, the Drinfeld coproduct defines a
$\mathcal U_q(\bs)$-module structure on $V$.
	
Let $V'$ be the $\mathcal U_q(\bs)$-submodule of $V$ generated by
$v_0\otimes v_+$.  Then $V'$ is a highest \(\ell\)-weight module of highest $\ell$-weight $\boldsymbol\Psi$.  We claim that $V'$ is simple.  Indeed, if
$N\subset V'$ is a non-zero submodule, by Lemma~\ref{submodule-structure}, there exists a subspace
$W_N\subset L(\boldsymbol\Psi_0)$, containing $v_0$, such that
\[
N=W_N\otimes L^{\bs}(\boldsymbol\Psi_+).
\]
In particular $v_0\otimes v_+\in N$.  Since $ V'$ is generated by
$v_0\otimes v_+$, we obtain $N=V'$.  Therefore
\[
V'\simeq L^{\bs}(\boldsymbol\Psi).
\]

Applying again Lemma~\ref{submodule-structure} to $V'$, there is a subspace
$W\subset L(\boldsymbol\Psi_0)$, containing $v_0$, such that
\[
V'=W\otimes L^{\bs}(\boldsymbol\Psi_+).
\]
Moreover $W$ is stable under the operators $x^+_{i,m}$ and
$\phi^+_{i,m}$, for $i\in I_0$ and $m\ge0$. Hence, using the
$q$-character formula for positive prefundamental Borel modules, we obtain
\[
\chi_q\bigl(L^{\bs}(\boldsymbol\Psi)\bigr)
=
\chi_q(W)\,\chi_q\bigl(L^{\bs}(\boldsymbol\Psi_+)\bigr) 
=
\chi_q(W)[\boldsymbol\Psi_+]
\prod_{i\in I_0}\chi_i^{\alpha_i(\mu)/\iota_i}.
\]

The following arguments are parallel to the untwisted case in
\cite{Hernandezshfted}. We include the proof for completeness.

Consider the restriction representations $(\widetilde{V'})_\mu$. By construction, $W\otimes v_+$ is stable under the action of $\tilde{x}^+_{i,m}$, $\phi^+_{i,m}$. Besides, the $\phi^+_{i}(z)$ has degree $0$ on $W$ ($W$ is a subspace of a $\U_q(\lgs)$-module). Therefore
\[
W\otimes v_+
\subset (V')_\mu,
\qquad
(W\otimes v_+)\cap (V')_{<\mu}=0.
\]
Thus we may view $W\otimes v_+$ as a subspace of
$(\widetilde{V'})_\mu$.

It remains to consider the negative currents.   For $i,j\in I_0$, $r\ge0$, and
$m>0$ with our choice of representatives, we have
\[
x^+_{j,r}x^-_{i,m}v_+
=
\delta_{ij}
\frac{\phi^+_{i,m+r}}{q_{i}-q_{i}^{-1}}v_+ .
\]
Since $\deg\Psi^+_i(z)=\alpha_i(\mu)$, we have $\phi^+_{i,m+r}v_+=0$ for all $m>\alpha_i(\mu)$. Then we have $x^-_{i,m}.v=0$ if $m>\alpha_i(\mu)$.
	
Using the Drinfeld coproduct, for $m>\alpha_i(\mu)$ and $w\in W$, we obtain
\[
x^-_{i,m}(w\otimes v_+)
=
\left(
\sum_{0\le r\le \alpha_i(\mu)}
\Psi^+_{i,r}x^-_{i,m-r}w
\right)\otimes v_+,
\]
where $\Psi^+_{i,r}$ is the coefficient of $z^r$ in $\Psi^+_i(z)$. This implies that $W\otimes v_1$ is stable of the action $\mathcal{U}^\mu_q(\hgs)$. Then we have a submodule $W\otimes v_+\subset (\widetilde{V'})_\mu$.

We now show that this submodule is generated by $v_0\otimes v_+$.  Let $w\otimes v_+\in W\otimes v_+$, there exists an element $x\in \mathcal U_q(\bs)^-$ such that $x(v_0\otimes v_+)=w\otimes v_+.$
Writing \(x\) in terms of Drinfeld negative current modes, define \(x^\sharp\)
	by replacing each \(x^-_{i,m}\) by
	\[
	D_{i,m}:=
	\sum_{0\le r\le \alpha_i(\mu)}
	\Psi^+_{i,r}x^-_{i,m-r}.
	\]
	The component of \(\Delta_D(x)(v_0\otimes v_+)\) whose right tensor factor has
	the same weight as \(v_+\) is precisely $	(x^\sharp v_0)\otimes v_+$.
	All other components have right tensor factor of weight strictly lower than the
	weight of \(v_+\). Hence $	w\otimes v_+=(x^\sharp v_0)\otimes v_+.$

This implies that 
\begin{align*}
W\otimes v_1&\subset (\langle D_{i,m}\rangle_{i\in I,m\in \ZZ}.v_0)\otimes v_1\\
&= (\langle D_{i,m} \rangle_{i\in I,m\ge \alpha_i(\mu)}.v_0)\otimes v_1
\subset \mathcal{U}^\mu_q(\hgs).(v_0\otimes v_1).
\end{align*}

Thus $W\otimes v_+$ is a highest $\ell$-weight
\(\mathcal U_q^\mu(\hgs)\)-module. We prove that it is simple. For the same argument, there is a nonzero polynomial $P(z)$ so that $P(z)x^+_i(z)=0$ on the finite-dimensional representation $L(\boldsymbol{\Psi_0})$. If there exists  primitive vector $w\otimes v_1$ such that $\tilde{x}^+_{i,m}(w\otimes v_1)=0$ for any $m\ge \alpha_i(\mu)$, then $x^+_{i,m}.w=0$ for any $i\in I_0,m\in \ZZ$.
Hence $w$ is a highest weight vector. Then we obtain $W\otimes v_+\cong L(\boldsymbol{\Psi})$ and complete the proof.
\end{proof}

\bibliographystyle{alpha} 
\bibliography{ref.bib}

\end{document}